\documentclass[11pt, twoside]{article}
\usepackage{amsfonts,amssymb,amsmath,amsthm}
\usepackage{color}
\usepackage{graphicx}
\usepackage[all]{xy}
 \usepackage{multirow}
\usepackage{pstricks,psfrag,xmpmulti,amscd, import}%%%HENON
\usepackage{enumerate}
\usepackage{epsfig}
\usepackage{graphics}
\usepackage[normalem]{ulem}

\usepackage{version}
 
\definecolor{OrangeRed}{cmyk}{0,0.6,1,0}            % half magenta only, full yellow
\definecolor{DarkBlue}{cmyk}{1,1,0,0.20}
\definecolor{myblue}{rgb}{0.66,0.78,1.00}
\definecolor{Violet}{cmyk}{0.79,0.88,0,0}
\definecolor{Lavender}{cmyk}{0,0.48,0,0}

\newtheorem{Conjecture}{Conjecture}

\newtheorem*{theorem}{Theorem}

 \divide\oddsidemargin by 2
\newtheorem{thm}{Theorem}[section]

\newtheorem{lem}[thm]{Lemma}

\newtheorem{prop}[thm]{Proposition}

\theoremstyle{definition}
\newtheorem{defn}[thm]{Definition}

\newtheorem{rem}[thm]{Remark}

\theoremstyle{remark}
\numberwithin{equation}{section}

\font\nt=cmr7

\def\note#1
{\marginpar
{\nt $\leftarrow$
\par
\hfuzz=20pt \hbadness=9000 \hyphenpenalty=-100 \exhyphenpenalty=-100
\pretolerance=-1 \tolerance=9999 \doublehyphendemerits=-100000
\finalhyphendemerits=-100000 \baselineskip=6pt
#1}\hfuzz=1pt}

\def\be{\begin{equation}}
\def\ee{\end{equation}}

\renewcommand{\epsilon}{\varepsilon}

\newcommand{\ra}{\rightarrow}

\newcommand{\diam}{\operatorname{diam}}

\renewcommand{\mod}{\operatorname{mod}}

\newcommand{\BB}{{\cal B}}
\newcommand{\CC}{{\cal C}}

\newcommand{\FF}{{\cal F}}
\newcommand{\GG}{{\cal G}}

\newcommand{\MM}{{\cal M}}

\newcommand{\PP}{{\cal P}}

\newcommand{\RR}{{\cal R}}
\renewcommand{\SS}{{\cal S}}

\newcommand{\ZZ}{{\cal Z}}

\newcommand{\C}{{\mathbb C}}

\newcommand{\D}{{\mathbb D}}

\newcommand{\N}{{\mathbb N}}

\newcommand{\Z}{{\mathbb Z}}

\def\B0{{\mathbf{0}}}

\newcommand{\ov}{\overline}

\renewcommand{\ra}{\rightarrow}

\newcommand{\Chat}{\hat{\C}}
\catcode`\@=12

\def\Empty{}
\newcommand\oplabel[1]{
  \def\OpArg{#1} \ifx \OpArg\Empty {} \else
  	\label{#1}
  \fi}
		
\renewcommand{\hat}{\widehat}

\newcommand{\s}{{\bf{s}}}

\newcommand{\Rat}{\operatorname{Rat}}
\renewcommand{\hat}{\widehat}

\renewcommand{\hat}{\widehat}

 \title{A survey on MLC, Rigidity and related topics}
 
\vspace{5cm}
\author{\small Anna Miriam Benini\thanks{Supported by  the Marie Curie IEF grant H2020 703269 COTRADY and by the SIR grant NEWHOLITE no. RBSI14CFME.} \\%\thanks{\r Partially supported by the ERC grant}\\  
 }

\begin{document}

\maketitle  
\begin{abstract} The MLC Conjecture states that the Mandelbrot set is locally connected, and it is considered by many to be the central conjecture in one-dimensional complex dynamics. Among others, it implies density of hyperbolicity in the quadratic family $\{z^2+c\}_{c\in\C}$. We describe recent advances on MLC and the relations between MLC, the Density of Hyperbolicity Conjecture, the Rigidity Conjecture, the No Invariant Line Fields Conjecture, and the Triviality of Fibers Conjecture. We treat  families of unicritical   polynomials and rational maps as well as the exponential family and  families of transcendental maps with finitely many singular values. 
\end{abstract}
\ 
%{\begin{small} 
%\tableofcontents
%\end{small}}
\section{Introduction} 
   
   The family of quadratic polynomials $\{z^2+c\}_{c\in\C}$, despite its simple (apparently innocent) form, exhibits a rich variety of dynamical behaviour. A special role is played by the unique \emph{critical value} $c$, which is the only point in $\C$ near which not all branches of the inverse are well defined and univalent.   The Mandelbrot set is a compact connected full subset of the plane, which can be characterized as the set for which the orbit of the critical value $c$ does not tend to infinity under iteration of the complex polynomial $z^2+c$. The apparently naive question of whether the Mandelbrot set is also locally connected- the so-called MLC Conjecture \cite{DH84}- has attracted work from excellent mathematicians in the past 30 years (see Section~\ref{The Mandelbrot Local Connectivity Conjecure} for the definition of local connectivity).     
      %, from the fields medalists  Arthur Avila, J. Milnor, C. McMullen and Jean-Christophe Yoccoz to other big names in the field of one-dimensional holomorphic dynamics including J. Graczyk, G. Levin,  M. Lyubich, C. Petersen, J. Kahn,  W. Shen, S. van Strien, M. Shishikura, G. \'Swiatek as well as D. Schleicher and L. Rempe-Gillen in the transcendental setting. 
 Good references on the subject of holomorphic dynamics at the graduate     level are \cite{Bea}, \cite{CG} and \cite{Mi}. At an undergraduate level you may look at \cite{Dev} and \cite{Roe}. More advanced texts, with an emphasis on renormalization, are \cite{Lyu99} and \cite{McM}. For transcendental maps, one can look at \cite{HY98}.

Two major breakthroughs occurred almost simultaneously in the Nineties: on one side Yoccoz showed that the Mandelbrot set is locally connected at all non-infinitely renormalizable parameters \cite{Hu}, and on the other side Lyubich and Graczyk and \'Swi\c{a}tek independently showed that hyperbolicity is dense in the family of real quadratic polynomials \cite{Lyu97}, \cite{GS97} (see also \cite{LvS}).

Most  approaches to the MLC conjecture, for example those involving the different types of renormalization,  are extremely technical, so we do not even try to describe the details of the proofs. 
The main goal of this survey is to offer an overview of the relations between the different conjectures related to the MLC conjecture, to picture the state of the art, and to give the  references for the details and for      most of the proofs. Some of this material naturally  overlaps with previous writings, especially with \cite{McM} and \cite{Hu}. However, we include: recent rigidity results by Levin for classes of maps with no a priori bounds (\cite{Lev11}, \cite{Lev14}); a section on  parabolic and near parabolic renormalization for quadratic polynomials (\cite{IS},\cite{CS});  results by  Avila, Kahn, Lyubich and Shen and by Cheraghi about rigidity and local connectivity for unicritical polynomials (\cite{AKLS}, \cite{KL09b}, \cite{Che10}); and some recent results on topological models \cite{BOPT}. We also present a description of the combinatorial approach to the MLC conjecture which is now widely used (\cite{Sc04},\cite{RS08}) and  a unified view of the relations between the main conjectures related to MLC.

We also   include a glimpse on families of transcendental maps with finitely many singular values, for which density of hyperbolicity is conjectured in appropriate finite-dimensional parameter spaces. In particular, the combinatorial approach to MLC conjecture and density of hyperbolicity (see Section~\ref{Fiber Section}) may have a chance to provide some more insight on the parameter spaces of transcendental maps with finitely many singular values.  The exponential family $\{e^z+c\}_{c\in\C}$ turned out to have striking dynamical and combinatorial similarities with the families  of unicritical polynomials. Recent advances on the description  of the escaping set for large classes of transcendental entire functions (see Section~\ref{Dynamic and parameter rays}) starts to shed some light on combinatorics of such maps  \cite{RRRS}, and offers new tools which may make it possible in the future to adapt some of the techniques used for polynomials and exponentials towards new methods to tackle the problem of density of hyperbolicity in the space of transcendental functions with finitely many singular values. 
 The reader is assumed to have some background in  complex dynamics. We keep the focus on quadratic polynomials but also give some information and references  about unicritical polynomial and families of transcendental and rational maps. 
 
It is very plausible  that the study of non-unidimensional parameter spaces of natural families of functions will be one of the main subjects of investigation in the next couple of decades, and that it is really useful to keep in mind a unified approach to   spaces of rational functions and finite dimensional spaces of entire transcendental functions. Many tools developed to study the MLC Conjecture also give insights on the fine structure of the parameter spaces of the functions under consideration. 

The targets of this survey are researcher interested in starting to work  on rigidity and related topics, and researchers  working in other areas of holomorphic dynamics who would like to get an overview of this very important topic, which so much shaped the research in one-dimensional holomorphic dynamics in the past 30 years. 

\subsection*{Structure of the paper}In Section~\ref{Preliminaries} we introduce singular values, the postsingular set, the families of functions that we will study in this paper and dynamic and parameter rays.  We also briefly write about wakes and about landing of dynamic and parameter rays. 
In Section~\ref{Density of Hyperbolicity} we define hyperbolic maps, structural stability, $J$-stability and the bifurcation locus. We state the appropriate theorems about density of structural stability in the parameter spaces that we are considering,  and state the Density of Hyperbolicity Conjecture.
 In Section~\ref{The Mandelbrot Local Connectivity Conjecure} we define local connectivity and describe results about local connectivity of the Mandelbrot set and of Julia sets, we mention the recent results by Levin and references for topological models of the Mandelbrot sets and of Julia sets. In Section~\ref{Topological rigidity} we define topological rigidity, quasiconformal maps, and quasiconformal rigidity. We also sketch Thurston's pullback argument. Section~\ref{The No Invariant Line Fields Conjecture (NILF)} is dedicated to defining invariant line fields, stating the No Invariant Line Fields conjecture  and sketching the proof that it is equivalent to the Density of Hyperbolicity Conjecture following \cite{McM}. In Section~\ref{Renormalization} we give the basic ideas behind the polynomial-like renormalization,  combinatorics, parabolic renormalization  and Yoccoz puzzle. In this section especially we strive for giving the  general concepts rather than the precise definitions, which are often slightly different from paper to paper. Finally in Section~\ref{Fiber Section} we describe  the Combinatorial Rigidity Conjecture, we define fibers and state the Triviality of Fibers Conjecture, describe an open-closed argument which also applies to the exponential family, and conclude with a  table of the different conjectures orbiting around the MLC conjecture.
   
\subsection*{Acknowledgements}   The author would like to thank  Misha Lyubich, Lasse Rempe-Gillen, Fabrizio Bianchi  for useful discussions on the topics of this paper, and Trevor Clark and Wolf Jung for several useful comments on the manuscript.

\section{Preliminaries} \label{Preliminaries}

\subsection*{Singular Values, Postsingular set and Escaping set}
Let $f$ be a holomorphic function.  The set of \emph{ singular values} $S(f)$ is the set of values near  which not all inverse  branches of $f$ are well defined and univalent. For rational maps, singular values coincide with  critical values. For transcendental maps, the set of singular values is the closure of the set  of critical and asymptotic values of $f$. Recall that an asymptotic value for a transcendental function $f$ is a point $a\in\C$ such that there exists a curve $\gamma:[0,\infty)\ra \C$ such that $|\gamma(t)|\ra\infty$ as $t\ra\infty$ while $f(\gamma(t))\ra a$ as $t\ra\infty$ (for example, $0$ is an asymptotic value for $e^z$). 

Entire transcendental  functions whose set of singular values is bounded are said to belong to the \emph{Eremenko-Lyubich class $\BB$}. 
In this survey we will mostly restrict ourselves to  the even smaller class of transcendental entire  functions with a finite set of  singular values, called the \emph{Speiser class} $\SS$ (or, functions of finite type).  %For this reason we will almost always restrict to rational functions and to transcendental functions with finitely many singular values. %Newhouse phenomena? % {\red Functions with $q$ singular values can be organized into families which are holomorphic spaces of dimension $q+2$, with the coordinates given by their $q$ singular values \cite{EL}.}   

The \emph{postsingular set} $P(f)$ of an entire function $f$ is defined as the closure of the orbits of all singular values, that is, 

\[P(f):=\ov{\bigcup_{s\in S(f),\ n\in\N} f^n(s)};\]
The $\omega$-limit set of a singular value is the set of accumulation points of its orbit.

For polynomials, all singular values are critical values, so it is customary to call $P(f)$ the \emph{postcritical set}. 

The \emph{escaping set} is the set $I(f)$ of points whose orbits tend to infinity. For polynomials, it coincides with the attracting basing of infinity, while for transcendental maps in class $\BB$ the set $I(f)$ is neither open nor closed and is contained in the Julia set. One can also show that  the Julia set is the closure of $I(f)$ (\cite{Er86}). It is important to keep in mind that for polynomials the orbit of each critical point is either bounded or escaping, while for transcendental maps singular orbits can be unbounded (for example, even dense in the plane) without escaping to infinity.

Essentially, in one-dimensional dynamics the non-linear behaviour of the function is due to the presence of singular value  and to the branching induced by the complexity of their orbits. 
A singular value is called \emph{recurrent} if it is contained in its own $\omega$-limit set, and \emph{non-recurrent} otherwise. When all singular values are non-recurrent, the function is usually simpler to study from the dynamical point of view, for example, it is easier to show local connectivity of Julia sets in the polynomial case or rigidity in the polynomial and exponential case.

\subsection*{Families of holomorphic maps}\label{Families}

In this survey we will be concerned with the family of unicritical poynomials, the family of rational maps of degree $d$, and  families of transcendental entire maps as defined in \cite{EL} whose parameter space can be represented by a finite-dimensional complex analytic manifold.
 
The most studied families of  polynomials are the one-dimensional  families of unicritical polynomials, which up to conjugation by affine maps can be written as $\{z^d+c\}_{c\in\C}$.  Their natural parameter space is the complex plane $\C$. The set of parameters in the family $\{z^d+c\}$ for which the Julia set is connected is called the \emph{Multibrot set} of degree $d$ and is denoted by $M_d$. When $d=2$, $M_2=M$ is the Mandelbrot set.

The exponential family $\{e^z+c\}_{c\in\C}$ is the simplest family of transcendental maps, and includes all entire transcendental maps with a unique singular value  up to conformal conjugacy. The complex plane  is usually considered as the  parameter space for this family although one should keep in mind that for any $c$, the  exponential  map  $e^z+c$ is affinely conjugate to all maps of the form $e^z+c+2k\pi i$ for $k\in\Z$ via the translation $z\ra z+2k\pi i$.
 
The space of all rational maps can be endowed with the structure of a  complex manifold $\MM_d$ of dimension $2d+1$. One often considers  the orbifold  $\Rat_d:=\MM_d/\sim$, where two rational maps are equivalent if and only if they are  conjugate by a Moebius transformation.  $\Rat_d$ is an orbifold of complex dimension $2d-2$. %(which equals the number of finite   critical values counted with multiplicity). 
For more on the space of rational maps of degree $d$ and their moduli space see \cite{Mi93}.

Natural families of transcendental maps with finitely many singular values can be defined, whose parameter spaces are a finite-dimensional complex analytic manifolds. Recall that $S$ is the class of entire functions with finitely many singular values.
By \cite{EL}, every function $g\in S$ with $q$ singular values is included in a (unique) finite dimensional complex analytic manifold $M_g$ of complex dimension $q+2$, and two functions which are topologically conjugate are included in the same manifold. To prove this one proceeds as follows. Given two functions $f,g$ with finitely many singular values one says that    $f\in M_g$ if and only if there exist homeomorphisms $\phi, \psi:\C\ra\C  $ such that $\psi\circ f=g\circ\phi$. This condition is called  \emph{topological equivalence}. It is then possible to show the following:  let $f_1=\psi_1\circ g \circ \phi_1\in M_g$ and  $f_2=\psi_2\circ g \circ \phi_2\in M_g$ such that there is an isotopy connecting   $\psi_1,\psi_2$ which fixes the  set of singular values for $f$ (and two additional values); then $f_1=f_2$.
 In particular, for any $f\in M_g$ the homeomorphism $\psi,\phi$ such that $\psi\circ f= g\circ\phi$ can be chosen to be quasiconformal %If $\psi, \phi $ can be chosen to be quasiconformal, then
 (in this case one says that $f,g$ are  \emph{quasiconformally equivalent}).  
This fact can be used to define an analytic structure on $M_g$ and show that it has dimension $q+2$; $M_g$ is the natural parameter space for the family of the map $g$.   {In this case one does not   consider equivalence classes under conjugation by affine maps like one does for unicritical polynomials or exponential maps, which explains the two extra parameters in dimension of the space.}
 %In this sense we can speak about \emph{natural families} of transcendental functions which are represented by finite-dimensional complex analytic manifolds.
  This finite dimensional complex analytic structure can be used to show that entire functions with finitely many singular values do not have wandering  domains (\cite{EL},\cite{GK}) and that the number of non-repelling cycles is bounded by the number of singular values  (\cite{EL}). 
%We observe that the fact that the space has dimension $q+2$ and not just $q$ is due to the fact that in this setting it is not customary to consider these maps up to conjugacion by an affine map, which would take into account the two extra   dimensions of the parameter space.

%One can see that in fact when $f,g$ belong to the same natural family $M_g$, the two homeomorphisms $\phi,\psi$ can be chosen to be quasiconformal, so that any two maps belonging to the same family $M_g$ are quasiconformally equivalent. [Maybe put this and shorten the discussion]

A surprising theorem which gives an idea of how much two functions belonging to the same natural family  share the same  dynamics near infinity is the following rigidity theorem, which can be also seen as an analogue of Boettcher Theorem for natural families of entire functions \cite{Re09}. 
 %We say that two functions $f, g\in\BB$ are quasiconformally equivalent near infinity if there exist quasiconformal maps $\phi,\psi:\C\ra\C$ such that $\psi(f(z))=g(\phi(z))$ whenever 
 %$|f(z)|$ or $|g(\phi(z))|$ is large enough. {\violet [Omit the near infinity business?]}
\begin{thm}[Conjugacy near infinity] Let $f, g\in\BB$ be quasiconformally equivalent.%near infinity.
 Then there exist $R>0$ and a quasiconformal map $\theta:\C\ra\C$ such that \[
\theta\circ f =g \circ\theta \text{\ \  on\ \  }J_R(f) :=\{z \in\C: |f^n(z)|>R \text{ for all } n>1\}.\]
Furthermore, $\theta$ has zero dilatation on $\{z\in J_R(f):|f^n(z)|\ra\infty\}$. 
\end{thm}

%{\red [remove?] Whenever we refer to a family $\FF$ of holomorphic maps- unless we put explicitly additional restrictions- we refer to one of the following  naturally defined families of holomorphich functions: the family of quadratic or unicritical polynomials ($\{z^2+c\}_{c\in\C} $ and $\{z^d+c\}_{c\in\C}$; the family of exponential maps $\{e^z+c\}_{c\in\C}$; the family of rational maps $\Rat_d$ of degree $d$;  natural families $M_g$ of transcendental functions with $q$ singular values as defined in \cite{EL}. All these families have an analytic structure: the first three  families are one-dimensional families, in the sense that functions in  these families are naturally parametrized by $\C$, in the other two cases they form analytic spaces of finite dimension $d-2$ and $q+2$ respectively.}
 The families of unicritical polynomials and the exponential family, being one-dimensional families,  have special features in that a combinatorial description is available not only for the dynamical plane  but also for the parameter plane (see Section~\ref{Fiber Section}). It is not clear to which extent it will be possible to extend this type of description to higher dimensional parameter spaces.

\subsection*{Dynamic and parameter rays}\label{Dynamic and parameter rays}
Dynamic rays (originally called \emph{external rays} in the polynomial setting) are injective curves in the escaping set which are naturally equipped with symbolic dynamics(  \cite{DH84}, \cite{DK}, \cite{DT86}, \cite{RRRS}, \cite{Ba}, \cite{DGH}). 
When they do not land alone, periodic dynamic rays  partition the dynamical plane and are especially meaningful for polynomials  and transcendental maps. 
Parameter rays are curves of escaping parameters which are defined in the parameter spaces of  unicritical polynomials and of the exponential family.  Dynamic and parameter rays are the base tools in all techniques used to prove rigidity which involve some version of Yoccoz puzzle (see Section~\ref{Renormalization}), that is, most of the work on rigidity-related topics of the last few decades.

Dynamic rays are proven to exist for polynomials \cite{DH84} and for entire functions in class $\BB$ which satisfy some additional growth condition (namely, they are a finite composition of functions of finite order in class $\BB$ \cite{RRRS}.) For the purpose of this paper, one can use the following definition of dynamic rays.

\begin{defn}[Dynamic rays] Let $f$ be a polynomial, or an entire transcendental function satisfying the growth conditions from \cite{RRRS}.
Then there exists a subspace $\SS\subset \Z^{\N}$, which is  invariant under the left-sided shift map $\sigma$, and  a maximal family $\GG$ for  $f$ of maximal injective curves $G_\s: (0,\infty)\ra I(f)\subset\C$ such that for each $\s\in\SS$ there exists a unique  curve $G_\s$ satisfying the following properties:
\begin{itemize}
\item for any $t>0,\  f^n_{[t,\infty)} \ra\infty$   uniformly  as $n\ra \infty$;
\item  $f(G_\s)=G_{\sigma\s}$;
\item $|G_\s(t)|\ra\infty $ as $t\ra\infty$.
\end{itemize}
We also assume that  $G_\s(t)\cap S(f)=\emptyset$. The family $\GG$ is called the family of \emph{dynamic rays} for $f$.
\end{defn}

One can prove  that for any escaping point $z$ there exists $n_0$ such that for all $n\geq n_0$ the point $f^n(z)$ belongs to a dynamic ray.
 
The family of dynamic rays is maximal in the sense that it is not possible to add more curves (in other words, $\SS$ is maximal), and the curves are maximal in that they cannot be extended further. The request of not containing singular values is not necessary but it makes things simpler in this context. 
 One can construct examples of entire functions in class $\BB$ which do not satisfy  the required growth condition and for which there are no curves in the escaping set \cite{RRRS}.

 \begin{defn} A dynamic ray is \emph{periodic} if it is labeled by a periodic sequence. A dynamic ray $G_\s$ \emph{lands} if and only if $G_\s(t)\ra z_0\in\C$ as $t\ra0$. A point $z_0$  is \emph{accessible} if there exists a ray landing at $z_0$.
 \end{defn} 
 
{To fix notation, we quickly review the construction of dynamic rays for polynomials with connected Julia set. 
For a polynomial $f$  of degree $d$ let   $K_f$ denote  its filled Julia set and assume it is connected. Equivalently, the postsingular set is bounded, or equivalently, if $f$ is unicritical  the parameter identifying $f$ belongs to the corresponding Multibrot set $M_d$. Since $K_f$ is a full compact set, there is a uniformization map $\phi:\C\setminus\ov{\D}\ra\C\setminus K_f$, which in addition is a conjugacy between $f$ and $z^d$ on $\C\setminus\ov{\D}$ by Boettcher's Theorem. Let   $\s$ be an angle written in $d$-adic expansion (there is some ambiguity in the choice of expansion when the angle is $\frac{1}{d^k}$, but this does not matter for us). The \emph{dynamic ray} $G_\s$ of angle $\s$ is defined as $G_\s:=\phi\{z=re^{2\pi i \s},\ r\in(0,\infty)\}$.  Since $\phi$ is a conjugacy between $f$ and the map $z\ra z^d$,  a ray $G_\s$ is a periodic set under $f$ if and only if $\s$ is a periodic sequence, and rays satisfy the functional equation $f(G_\s)=G_{\sigma\s}$, with $\sigma$ the shift map on sequences.
A similar definition can be given also when $K_c$ is not connected. In this case, some critical  values belong to dynamic rays. 

For transcendental maps (\cite{RRRS}, \cite{DT86},\cite{Ba},\cite{BF1}), dynamic rays are  constructed using symbolic dynamics. They are  labeled by a sequence $\s\in\SS\subset \Z^\N$, which   in this context is usually  called \emph{address} and plays the role of the angle. There is no  precise characterization of  $\SS$ for all families of transcendental functions but $\SS$ contains all periodic sequences (\cite{Re08}, \cite{Ba}, \cite{BRG}). Several   conditions for an address to be realized can be found in \cite{ABR}. 

For unicritical polynomials and for the exponential family we can define the \emph{parameter ray} $\GG_\s$ as the set of parameters which belong to the dynamic ray $G_\s$ in their own dynamical plane.

The following theorem is due to Douady and Hubbard for polynomials \cite{Hu}; to \cite{Re06}, \cite{BL} for exponential maps and to \cite{Re08}, \cite{De},  \cite{BRG} for maps in class $\BB$.

\begin{theorem}[Douady-Hubbard landing theorem]\label{Douady Hubbard} Let $f$ be  polynomial or a transcendental entire map in class $\BB$ for which dynamic rays exist. Assume $\PP(f)$ is bounded. Then  every (pre)periodic dynamic ray lands at a repelling (pre)periodic or parabolic (pre)periodic point, and every repelling or parabolic (pre)periodic point is the landing point of at least one and at most finitely many  (pre)periodic dynamic ray.
\end{theorem}  

An analogous theorem holds for parameter rays \cite{DH84}, \cite{Sc99}, \cite{Sc04}:
 
\begin{thm}[Landing of parameter rays at parabolic and Misurewicz parameters]
If $\FF$ is 
the exponential family or a family of unicritical polynomials,  every parameter ray with  periodic or preperiodic address lands at a parabolic or Misiurewicz parameter respectively. Conversely, every parabolic or Misurewicz parameter is the landing point of at least one and at most finitely many parameter rays with periodic or preperiodic address.
\end{thm}

   The following theorem is classical for polynomials (\cite{DH84}, \cite{Mi}) and holds for exponentials as well (\cite{Re06}, \cite{Sc99}). 

\begin{thm}[Ray portraits are born in wakes]\label{Ray portraits are born in wakes}  Every cycle of periodic dynamic rays has two characteristic addresses/angles $\s_+, \s_-$; the parameter rays of addresses/angles  $\s_+, \s_-$ land together at a parabolic parameter and hence partition the plane into two regions, one of which does not contain the unique period one hyperbolic component and which is called a \emph{parabolic wake} $W_{\s_-,\s_+}$. The two dynamic rays of address/angle $\s_+, \s_-$ land together for a map $f_c\in\FF$ if and only if  $c\in W_{\s_-,\s_+}$.
\end{thm}

An analogous theorem holds for sets of preperiodic dynamic rays landing at the same point and parameter rays with preperiodic angles/addresses landing at Misurewicz parameters and forming Misiurewicz wakes.

Let us close this section with some remarks on 'ghost limbs' or 'irrational subwakes'. It is known that for unicritical polynomials, every non-escaping parameter either belongs to the boundary of a hyperbolic component or  is  contained in the wake of such a component  (see Theorem 2.3 in \cite{Sc04} for details). In particular, there cannot be 'decorations' outside of the wakes attached, for example, at Cremer or Siegel parameters (the 'ghost limbs'). The proof of this results relies  on the fact that for polynomials with non-escaping singular value, every repelling periodic point is the landing point of at least one periodic ray, and then uses  Theorem~\ref{Ray portraits are born in wakes} to get the claim. {Even for the exponential family it is not yet known in full generality  whether every repelling periodic point is the landing point of at least one periodic ray whenever the singular value is not escaping.} So the presence of   ghost limbs is not excluded yet;  although, if ghost limbs exist they cannot contain parameters with bounded postsingular set by Theorem~\ref{Douady Hubbard}.  

\section{Density of Hyperbolicity, Structural Stability and the Bifurcation Locus}\label{Density of Hyperbolicity}
From now on we restrict attention to rational maps, polynomials, and transcendental functions with finitely many singular values. 
We call  a map  with finitely many singular values  \emph{hyperbolic} if and only if all its  singular values belong to attracting basins.  The parameter identifying a hyperbolic map in an appropriate parameter space is called a hyperbolic parameter. A \emph{hyperbolic component} in a parameter space is a maximal connected component of hyperbolic parameters.

 Several equivalent definition of hyperbolicity  for rational maps can be found in \cite[Theorem 3.13]{McM}. {The dynamics of hyperbolic maps is very well understood.  Among other things, the Julia set of hyperbolic  polynomial  and rational maps have measure zero and are hyperbolic. Under some extra condition (7.1) hyperbolic maps in class $S$ also have Julia set of zero area (see Thm 8 in \cite{EL}). 

The most important consequence of the MLC conjecture is that it implies density of hyperbolic quadratic polynomials in their parameter space $\C$, or equivalently in the Mandelbrot set. The Density of Hyperbolicity Conjecture for rational maps goes back to Fatou \cite{Fa}:
 \begin{Conjecture}[Density of Hyperbolicity Conjecture]\label{Density of hyperbolicity conjecture} Let $\FF$ be  the  family of unicritical polynomials of degree $d$, the family of rational map of degree $d$, the exponential family or a family of entire maps with finitely many singular values as defined in \cite{EL}. Then hyperbolic maps are dense in $\FF$. 
 \end{Conjecture}

%{\violet Add the theorem from McMullen? Need to check the proof for transcendental}
We now describe $J$-stability, structural stability, and their relation with Conjecture~\ref{Density of hyperbolicity conjecture}.  
Let $\FF$ be one of the families from Section~\ref{Families}. The set of \emph{structurally stable maps} $\SS(\FF)$ is defined as 
\[\SS(\FF):=\{f\in\FF: f \text{ is topologically conjugate to $f$ for all $g$ in a neighborhood of $f$}\}.\] 

Parameters with an indifferent cycle and parameters with non-persistent critical relations (an example of non-persistent critical relation is   the condition $f^{k+n}(s)=s$ for $s\in\SS(f)$ and for some $k,n\in\N$) cannot be structurally stable, since in both cases they can be perturbed to destroy the indifferent cycle or the critical relations. 
When $\FF$ is the family of quadratic polynomials, $\SS(\FF)$ is naturally identified with $\C\setminus(\partial M\cup \CC)$, where $\CC$ is the set of superattracting parameters (that is, the centers of hyperbolic components).

Structural stability is an important notion for general dynamical systems. However, in holomorphic dynamics, a more widely used  notion of stability is the notion of $J$-stability. 
To define $J$-stability we need to define the concept of holomorphic motion. 
\begin{defn} Let $\Lambda$ be an open subset of a complex manifold with a marked point $\lambda_0$. A  \emph{holomorphic motion} of a set $X\subset\C$ over  $\Lambda$ with center at $\lambda_0$ is a map $\phi: \Lambda\times X\ra\C$ such that:
\begin{itemize}
\item The map $\lambda\mapsto \phi(\lambda, z) $ is holomorphic in $\lambda$ for all $z\in X$;
\item The map $z\mapsto \phi(\lambda, z) $ is injective for every $\lambda \in\Lambda$;
\item $\phi_{\lambda_0} $ is the identity.
\end{itemize}
\end{defn}
%In general, $\Lambda$ is a simply connected open subset of  $C^n$, or even not simply connected if one is willing to take holonomies into account.
 In our case $\Lambda$ is an open  subset of the parameter space, that is a subset of $\C$, Rat$_d$ or $M_g$, and the parametrization is chosen so that $f_\lambda$ depends holomorphically on $\lambda$.  

We will use several times the following crucial property of holomorphic motions, called the  \emph{$\lambda$-lemma}  (\cite{MSS}, \cite{Lyu83}; see Section~\ref{Topological rigidity} for a definition of quasiconformality).
\begin{lem}[ $\lambda$-lemma ]\label{Lambda Lemma}Any holomorphic motion of a set $X$ can be extended (uniquely) to a holomorphic motion of the closure $\ov{X},$ and that the map $\phi_\lambda:\ov{X}\ra\C$ is quasiconformal for any $\lambda\in\Lambda$.
\end{lem} 
For example, a holomorphic motion of the set of periodic points which is defined over a definite open subset of the parameter space extends to a holomorphic motion of the entire Julia set. A much stronger version by Slodkowski \cite{Sl91}  states that in fact the  holomorphic motion can be extended  to a motion of all of $\C$.

A map $f_0$ is \emph{$J$-stable} if the Julia set moves holomorphically at $f_0$. This means that  for all $f_\lambda $ in a neighborhood $\Lambda$ of $f_0$, $f_0$ is conjugated to $f_\lambda$ on the Julia set $J(f_0)$ by a homeomorphism $\phi_\Lambda$, and that the maps $\phi_\lambda:J(f_0\ra J(f_\lambda))$ are a holomorphic motion of $J(f_0)$ over $\Lambda$. To give an example, superattracting parameter in the Mandelbrot set are $J$-stable but not structurally stable. 
  There are several equivalent definitions of $J$-stability \cite[Theorem 4.3]{McM}:
\begin{thm}\label{Equivalent notions of J-stability} Let $\FF=\{f_\lambda\}_{\lambda\in\Lambda}
$ be a family of holomorphic maps as in Section \ref{Families}, $\Lambda$ be a manifold parametrizing $\FF$. The following are equivalent:
\begin{enumerate}
\item The number of attracting cycles of $f_\lambda$ is locally constant at $\lambda_0$.
\item The maximum period of an attracting cycle of $f_\lambda$ is locally bounded at $\lambda_0$.
\item The Julia set moves holomorphically at $\lambda_0$.
\item For all $\lambda$ close to $\lambda_0$ every periodic  point of $f_\lambda$
 is attracting or repelling.% or persistently indifferent. 
 \end{enumerate}
 Suppose in addition that $s_i:M\ra\hat{\C}$, $\lambda\mapsto s_i(\lambda)$ are holomorphic maps parametrizing the finitely many singular values of $f_{\lambda_0}$. 
 Then the following conditions are also equivalent to $J-$stability.
 \begin{enumerate}
 \item[5.] For each $i$, the function $\lambda\mapsto f^n_\lambda(s_i(\lambda))$ form a normal family.
 \item[6.] There is a neighborhood $U$ of $\lambda_0$ such that for all $\lambda\in U$, $s_i(\lambda)\in J(f_\lambda)$ if and only if $s_i(\lambda_0)\in J(f_{\lambda_0})$.
 
 \end{enumerate}
\end{thm}
\begin{rem}The proof for transcendental entire functions is the same as in \cite{McM}, which uses a more general definition of families of rational maps.{ Indeed in his cases there can be  families with a persistently indifferent cycle, which cannot happen for the families that we are considering (see Lemma 6 in \cite{EL}).}
\end{rem}

The following theorem is proven in \cite{MSS} for polynomials and  rational maps and in \cite[Theorem 9]{EL} for families $M_g$ of transcendental functions in class $S$.

\begin{thm}[Density of J-Stability]\label{Density of J-Stability} Let $\FF$ be  the  family of unicritical polynomials of degree $d$, the family of rational map of degree $d$, the exponential family  or a family $M_g$ of transcendental entire maps as defined in \cite{EL}. 
Then $J$-stable maps are  dense in $\FF$.
\end{thm}

Clearly, structural stability implies $J$-stability. The opposite is not true, since, for example,   superattracting cycles are not stable under perturbation. In fact, structurally stable functions are  those J-stable functions which do not have non-persistent critical relations.  
Using density of J-stability and the fact that the sets of parameters satisfying non-critical relations are not 'too big', one can also show density of structural stability (see e.g. \cite{EL}, Theorem 10). 
  
The complement of the set of $J$-stable parameters is called the \emph{bifurcation locus}. For quadratic 
polynomials, the bifurcation locus is the boundary of the Mandelbrot set $M$,  for unicritical  
polynomials, the bifurcation locus is the boundary of the Multibrot set $M_d$. For the exponential family the bifurcation locus  is harder to describe, is non-locally connected,  and contains  all escaping parameters. Little is known about parameter spaces for $\Rat_d$  and next to nothing is known for parameter spaces of transcendental families.

{It is easy to check that non-superattracting hyperbolic  parameters are structurally stable (it is sufficient to perturb the multiplier of the finitely many attracting cycles),} so given Theorem~\ref{Density of J-Stability} it is natural to ask  whether hyperbolic parameters are dense as well, or whether there are components of the structurally stable set which are not hyperbolic. Such components, if they exist, are called \emph{non-hyperbolic components} or \emph{queer components}.  

One may wonder about what the dynamics would be like for a function belonging to a non-hyperbolic component.  
For unicritical polynomials and exponential maps, being in a non-hyperbolic component implies that there are no Siegel, parabolic and Cremer cycles (those parameters are all in the boundary of hyperbolic components, since the unique indifferent cycle can be perturbed to become attracting) nor attracting cycles in $\C$. So polynomials in non-hyperbolic components have filled Julia sets with empty interior, and exponential maps  in non-hyperbolic components have the Julia set which equals the entire plane. It is less clear what happens to the dynamics in non-hyperbolic components when there is more than one free singular value. By Theorem~\ref{Equivalent notions of J-stability}, a map in the structurally stable locus only has attracting and repelling periodic points. However  a  priori, in a non-hyperbolic component there could be  attracting cycles which coexist with other behaviours of the other singular values.

Clearly, Conjecture \ref{Density of hyperbolicity conjecture} is equivalent to the statement that there are no non-hyperboic components.%, so if it fails, there is a fairly big set  of parameters (namely, an open set) for which it fails.

Density of hyperbolicity has been proven in the space of real quadratic polynomials in \cite{Lyu97} and \cite{GS97}, and in the space of all real polynomials in \cite{KSvS07} (under the assumption that critical points are real). {It has been proven for some classes of transcendental entire maps in \cite{RvS15}.

On the other hand, Epstein and Rempe-Gillen \cite[Theorem 1.1]{ERG} have recently shown the existence of a Newhouse phenomenon for functions with bounded set of singular values, that is, the existence of open regions in the appropriate parameter space   which are not in the structurally stable set. So restricting to  families of transcendental maps with finitely many singular values is necessary for this discussion.

\section{The Mandelbrot Local Connectivity Conjecure (MLC),  Local Connectivity of Julia Sets and topological models}\label{The Mandelbrot Local Connectivity Conjecure}

There are several definitions of local connectivity. Here we refer to the following one \cite[p.182]{Mi}:
\begin{defn} A compact connected set $X\subset \Chat$ is locally connected at a point $x$ if and only if $x$ has a basis of connected (not necessarily open) neighborhoods in $X$. The set $X$ itself is locally connected if and only if it is locally connected at every point $x\in X$. 
\end{defn} 

\begin{Conjecture}[MLC]
The Mandelbrot set is locally connected.
\end{Conjecture}

  Yoccoz \cite{Hu} proved in the Nineties that the Mandelbrot set is locally connected at many parameters, including  every parameter that is at most finitely renormalizable, and that the corresponding Julia sets are locally connected as well (see Section~\ref{Renormalization} for a definition of renormalization and for the statement of Yoccoz's Theorem). Subsequent works by  Kahn and Lyubich   proved local connectivity also for several classes of infinitely renormalizable parameters (\cite{Kah06}, \cite{KL08}, \cite{KL09}).

Local connectivity of Julia sets of polynomials and of the Multibrot sets is related to landing properties of dynamic and parameter rays (and hence, to the dynamics of the map under consideration) through the following classical theorem (see e.g. Theorem 17.14 in    \cite{Mi}).

\begin{thm}[Carath\'eodory-Torhost's Theorem] Let  $X\subset \Chat$ be a  compact connected full set and let  $\phi$ be the Riemann map from $\C\setminus\D\ra\C\setminus X$. Then $\psi$ extends as a homeomorphism to $\ov{\D}$ if and only if $\partial X $ is locally connected.
\end{thm}

{So, all parameter rays lands (with the landing point depending continuously on the angle) if and only if $M$ is locally connected. Similarly, given a polynomial $P$, all dynamic rays in the dynamical plane  of $P$ land (and the landing point depends continuously on the angle) if and only if $J(P) $ is locally connected. See also Chapter 18 in \cite{Mi} for more on local connectivity versus landing of rays.}

Rays which land together  in the dynamical plane for a polynomial  $P$ (respectively  in the parameter plane of a family of unicritical polynomials) induce an equivalence relation  on $\partial \D$, defined by $\theta\sim\theta'$ if and only if the dynamic (respectively parameter) rays of angle $\theta$ and of angle $\theta'$ land together at the same point in the Julia set of $P$ (respectively, on the boundary of the corresponding Multibrot set).  Equivalent points on $\partial\D$ can then be connected by geodesics in $\D$, and some work can be done in order to construct a closed lamination (which, in the dynamical case, is also endowed with a natural dynamics on the leaves).    Using this and assuming local connectivity, one can build  abstract models for $M$ (respectively for a Julia set $J$) which, under the assumption of local connectivity,  are   homeomorphic to the Julia set (respectively to the Multibrot set), and check properties (like for example   density of hyperbolicity) in the abstract model. Following this strategy,   Douady in  \cite{DH84} constructed a topological model of the Mandelbrot set to show  that MLC implies density of hyperbolicity  (see also the proof in \cite{Sc04}).  
Two general strategies to construct topological models for connected, full, locally connected  compact sets have been introduced by Douady \cite{Do}. See also  the approach to topological models of the Mandelbrot set which uses Thurston's minor laminations \cite{Th85}. 
 Recent work has been done on constructing topological models for Julia sets of  cubic polynomials \cite{BOPT} and rational maps \cite{Ro}.  

In Section~\ref{Fiber Section}  we will sketch a different proof of the fact that MLC implies density of Hyperbolicity, following \cite{RS08} (see also \cite{Sc04}). Indeed one can  show that MLC is equivalent to Triviality of Fibers (or, equivalently, Combinatorial Rigidity) and that the latter implies Density of Hyperbolicity.

\subsection*{Locally and Non Locally Connected Julia Sets for Quadratic Polynomials}
It is known that many Julia sets of quadratic polynomials are locally connected, for example  %that is, the limit of the period doubling bifurcation cascade 
 the Julia sets of real quadratic polynomials, finitely renormalizable quadratic polynomials with no indifferent periodic points, the Julia set corresponding to the Feigenbaum parameter $c_F$, etc (see \cite{LvS}, \cite{Lyu97}, \cite{KvS09}, \cite{Pe96},\cite{PZ}, \cite{Hu}). See \cite{KL09b} for results on local connectivity of Julia sets for unicritical polynomials of degree $d>2$, and \cite{KL09}  for local connectivity of Julia sets of  finitely renormalizable unicritical polynomials with all periodic points repelling. See \cite{Ro} for results about local connectivity of Julia sets for rational maps obtained via Newton's method. 

  However,   Julia sets can be non-locally connected   even for quadratic polynomials.   The first examples of non-locally connected
Julia sets was given by Sullivan and Lyubich by showing that every polynomial with a Cremer point has a non-locally connected Julia set (\cite{Mi}, Theorem 18.5).

{One can see that connected Julia sets of unicritical polynomials are locally connected if and only if they contain no wandering continua,  and if and only if the Yoccoz puzzle pieces shrink to points \cite{Lev98}; see Section~\ref{Renormalization} for the definition of Yoccoz puzzle}. 

In Yoccoz's paper and in many of the subsequent work in the field, local connectivity of Julia sets is often a preliminary step in proving local connectivity of the Mandelbrot set at a given parameter. One uses shrinking of puzzle pieces in the dynamical plane- where the dynamics helps in getting control on the size of puzzle pieces- to show shrinking of parapuzzle pieces in the parameter plane (see Section~\ref{Renormalization}). However the two things are not   as necessarily related as one may expect. {For example, for   the Feigenbaum parameter $c_F$ one knows that the  Julia set %for the map  $z^2+ c_F$ 
 is locally connected, but it is not known whether the Mandelbrot set itself is locally connected at $c_F$. } 
 On the other hand, using parabolic bifurcations it is possible to construct  (\cite{Mi00}, \cite{So00}) non-locally connected Julia sets which correspond to infinitely renormalizable parameters. {The parameters of such Julia sets are constructed as limits of parabolic bifurcations whose periods  tend to infinity quickly enough. As observed in \cite{Lev11}, using the bound on the size of limbs of the Mandelbrot set following from the Pommerenke-Levin-Yoccoz inequality, one can  choose the sequence of parabolic bifurcations in such a way  that the Mandelbrot set is locally connected at the limiting parameter.}

 Levin (\cite[Theorem 2]{Lev11},\cite{Lev14}, \cite{Lev09})  explores conditions under which the Mandlebrot set is locally connected at parameters obtained from the construction in \cite{Mi00}, by making the construction  quantitative.  A very interesting feature of the work by Levin is that it does not use the complex bounds introduced by Sullivan \cite{Su88}, since the latter do not hold in general for the  class of examples that he is considering. This is essentially due to the fact that the small Julia sets do not shrink to points (see Section~\ref{Renormalization} and \cite{Lev11} for more on the relation between complex bounds and small Julia sets), which implies lack of complex bounds.  Levin's approach is based on an extension result for the multiplier of a periodic orbit  beyond the domain where it is attracting, and on a strenghtening of the Pommerenke-Levin-Yoccoz inequality in \cite{Lev09}.

\begin{rem} 
Local connectivity does not seem to be the right notion to look at when considering Julia sets or parameter spaces of transcendental entire maps.  Indeed, the bifurcation locus  for the exponential family is not locally connected (see Theorem 5 in \cite{RS08}, which follows a strategy used  by Devaney to show that $e^z$ is not structurally stable). However, hyperbolic maps are still expected to be dense. Also, when the Julia set is not the entire plane, 
in many cases there are curves in the Julia set which form so called \emph{Cantor bouquets} which are not locally connected. However, topological models (called \emph{pinched Cantor bouquets})do make sense at least for hyperbolic maps \cite{Mi12}.
\end{rem}

\section{Topological rigidity and Quasiconformal rigidity} \label{Topological rigidity}
We say that two maps in the same family (as defined in Section~\ref{Preliminaries}) belong to the same {\emph topological, quasiconformal or combinatorial class} if they are in a connected subset of the parameter space in which all maps are topologically conjuigate, quasiconformally conjugate, or combinatorially equivalent as defined in Section~\ref{Combinatorial rigidity section}.
The philosophy beyond the concept of rigidity is to show that topological classes, quasiconformal classes, or combinatorial classe are in fact singletons, unless they consist of hyperbolic maps. This concept is stemming from the observation that parameters in a connected component of the structurally stable set (that is, a hyperbolic or non-hyperbolic component) clearly belong to the same topological, quasiconformal and combinatorial class.
 In this section we deal with topological and quasiconformal classes, while combinatorial classes are treated in Section~\ref{Combinatorial rigidity section}.
Unless explicitly stated otherwise, $\FF$ will denote any family as in Section~\ref{Preliminaries}.
\begin{defn}
We say that a map $f$ in a family $\FF$ is \emph{topologically rigid}   if it is not topologically  conjugate to any other map in a neighborhood of $f$ inside $\FF$.
\end{defn}

%\begin{Conjecture}[Topological Rigidity Conjecture \cite{Lyu93}]Any non-hyperbolic polynomial is topologically rigid, that is, not topologically conjugate to any other polynomial.
%\end{Conjecture}
For a statement of the following conjecture for the family of quadratic polynomials see  \cite{Lyu97}.
\begin{Conjecture}[Topological Rigidity Conjecture] Let $f\in\FF$ be non-hyperbolic.  Then  $f$ is topologically rigid, that is, not topologically conjugate to any other $g\in\FF$ with $g$ in a sufficiently small neighborhood of $f$.
\end{Conjecture}

Since functions in the structurally stable locus are topologically conjugate to all functions in a neighborhood, but a priori a function could be topologically conjugate to some functions in a neighborhood and not to all of them,  the Topological Rigidity Conjecture is formally slightly stronger than the Density of Hyperbolicity Conjecture. %top classes are open 

To define quasiconformal classes we first define quasiconformal maps.  
There are several equivalent definitions (\cite[Section 2.6]{McM}, \cite{BrFa}).
%\begin{defn} An orientation-preserving homeomorphism $\psi:\C\ra\C$ is K-\emph{quasiconformal} if there exists $K\geq1$ such that for any annulus $A\subset \C$ 
%\[\frac{1}{K}\mod A\leq \mod \psi(A)\leq K\mod A.\]
%\end{defn}
%If $K$ can be taken to be 1, then $\psi$ is conformal.
%Equivalently, $f$ is K- quasiconformal if and only if it admits partial distributional  derivatives in $L^1_{\text{loc}}$ satisfying 
%\[\frac{\ov{\partial}_z \psi}{{\partial}_z \psi}=:\mu(z)\frac{d\ov{z}}{dz}\]
%with $\|\mu\|_\infty\leq\frac{K-1}{K+1}$. 
%The form $\mu(z)\frac{d\ov{z}}{dz}$  is  a \emph{Beltrami differential}, that is a $(-1,1)$ form with bounded essential supremum and measurable coefficient.

\begin{defn} An orientation-preserving homeomorphism $\psi:\C\ra\C$ 
$\psi$ is K- quasiconformal if and only if it admits partial distributional  derivatives in $L^1_{\text{loc}}$ satisfying 
\[\frac{\ov{\partial}_z \psi}{{\partial}_z \psi}=:\mu(z)\frac{d\ov{z}}{dz}\]
with $\|\mu\|_\infty\leq\frac{K-1}{K+1}$. 
The form $\mu(z)\frac{d\ov{z}}{dz}$  is  a \emph{Beltrami differential}.
\end{defn}
More visually, $\psi$  is K-\emph{quasiconformal} if there exists $K\geq1$ such that for any annulus $A\subset \C$ 
\[\frac{1}{K}\mod A\leq \mod \psi(A)\leq K\mod A.\]
If $K$ can be taken to be 1, then $\psi$ is conformal.
%Equivalently, , that is a $(-1,1)$ form with bounded essential supremum and measurable coefficient.
 The geometric interpretation of $\mu$ (see \cite{BrFa} for a very clear explanation) is that it defines a field of ellipses on the tangent space (which in this case can be identified with $\C$ itself) where the ratio between the major and minor axis is bounded by $K$ and the direction of the major axis depends only on  $\arg(\mu)$. This field of ellipses is called an \emph{almost complex structure} (induced by $\psi$). The \emph{standard complex structure} $\sigma_0$ is when almost every ellipse is in fact a circle, or equivalently $\mu=0$ almost everywhere. 
 
  Weyl's Lemma ensures that if for a quasiconformal map $\psi$ we have that  $\ov{\partial}_z\psi=0$ almost everywhere %, or equivalently the almost complex structure induced by $\psi$ is $\sigma_0$,
    then $\psi$ is conformal. 
  
  The Measurable Riemann Mapping Theorem (see Section 1.4 in \cite{BrFa}) states that given any Beltrami differential $\mu$ there exists a quasiconformal map $\phi$ which integrates $\mu$, that is, $\phi_*\mu=\sigma_0$ the standard conformal structure, and  $\phi$ is unique up to postcomposition with a conformal map. 
Moreover, if one considers a family of Beltrami differential $\mu_t$ depending continuously (or holomorphically) on a parameter $t$ one obtains a family of quasiconformal maps depending continuously (or holomorphically) on $t$.}

\begin{defn}
We say that a map $f\in\FF$ is \emph{quasiconformally rigid} if it is not quasiconformally   conjugate to any other map in a neighborhood of $f$ inside $\FF$.
\end{defn}

\begin{Conjecture}[Quasiconformal Rigidity Conjecture]Let $f\in\FF$ be non-hyperbolic.  Then  $f$ is quasiconformally rigid, that is, not quasiconformally  conjugate to any other $g\in\FF$ with $g$ in a sufficiently small neighborhood of $f$.
\end{Conjecture}

\begin{rem} Since $J$-stable parameters are dense in the families that we are considering, the bifurcation locus has empty interior. By the Measurable Riemann Mapping Theorem  
 quasiconformal  classes are either open or  singletons, so quasiconformal classes for parameters in the bifurcation locus are singletons.
\end{rem}
 As one would expect, quasiconformal classes of hyperbolic parameters are not singletons, and  coincide with hyperbolic components minus the centers.
%
%\begin{prop}[Quasiconformal classes of hyperbolic parameters] Consider a family of unicritical polynomials or the exponential family. If two parameters  belong to the same hyperbolic component $W$ and neither of them is   superattracting, they are quasi-conformally conjugated. 
%\end{prop}
Recall that for functions with finitely many singular values a \emph{hyperbolic component} is a maximal  connected component of the parameter space in which all the singular values belong to attracting basins.  
\begin{prop}[Quasiconformal classes of hyperbolic parameters] Consider a hyperbolic component for a family $\FF$ as in Section~\ref{Preliminaries}. If two parameters  belong to the same hyperbolic component $W$ and neither of them is   superattracting, they are quasi-conformally conjugated. 
\end{prop}
\begin{proof}[Sketch of proof following Theorem 4.12 in \cite{Lyu99}] Recall first 
that transcendental maps with finitely many singular values do not have wandering and Baker domains, so that  $\C$ is the union of the Julia set,  the finitely many immediate attracting basins, and the preimages of the latter. For polynomials, one also needs to add the superattracting basin of infinity.
 
Since parameters in hyperbolic components have no indifferent cycles, all of their periodic points can be continued analytically and without colliding  over the entire hyperbolic component, giving a holomorphic motion of the set of repelling periodic points which respects the dynamics. This motion can be extended to the entire Julia set using Ma\~n\'e-Sad-Sullivan's $\lambda$-lemma and using the fact that the Julia set is the closure of repelling periodic points.
 Using linearizing coordinates (and pullbacks), for each  immediate  basin of attraction one can  construct a  holomorphic motion of the basin of attraction  which is a conjugacy. For polynomials,  one also needs to construct a  holomorphic motion of the basin of attraction of infinity  using Boettcher's Theorem, which again is a conjugacy by construction. In this way we constructed a holomorphic motion of the entire complex plane over a simply connected open set containing the two parameters, which is automatically continuous -and in fact  quasiconformal- by the $\lambda$-lemma.   When hyperbolic components are known to be simply connected as in the exponential and unicritical polynomial case, the holomorphic motions obtained when varying the starting parameters can be patched together to obtain a  holomorphic motion over the entire hyperbolic component  minus the center.
\end{proof}

\begin{rem}\label{Non-hyperbolic qc classes}
With a similar proof, one can show that for unicritical polynomials and for   exponentials any two  parameters in a non-hyperbolic component are quasiconformally conjugate.  See also the proof of Theorem 4.9 and Corollary 4.10 in \cite{McM}. 
%However one  cannot say anything for other families because functions in non-hyperbolic components may have   Fatou set with non-empty interior, hence the holomorphic motion of the repelling periodic points does not extend to all of $\C$, or it could happen that periodic points become indifferent inside the component [can it? check again McM thm about structural stability].
\end{rem}

The analogue of quasiconformal rigidity for real one-dimensional families of analytic maps is \emph{quasi-symmetric rigidity}.  Substantial work in the direction of quasisymmetric rigidity, as well as an exhaustive introduction to the previous state of the art, can be  found  in \cite{CvS}. 
 
\subsection*{Thurston's Pullback Argument to construct quasiconformal conjugacies}
A quite standard technique which is used to show rigidity is to construct quasiconformal conjugacies between combinatorially equivalent maps by using a pull-back argument, and then using an open-closed argument (see Section~\ref{Combinatorial rigidity section} and  \cite{Che10}) to deduce that combinatorial classes consist of a single point. 
{The Pullback Argument  was probably introduced by Thurston for postcritically finite maps, and has been used several times since then (see among others  \cite[Paragraph 11]{Su88}, \cite{Lyu97} and  \cite{Che10}, \cite{KSvS07} for polynomials, \cite{Cui01} for rational maps, \cite{Be15} for exponential maps.} In order to be able to take pullbacks, one needs some information on the behaviour of the postsingular set- ideally, the singular value is non-recurrent, but one can also deal with weak forms of recurrence. 

This  is a sketch of how  to construct a quasiconformal conjugacy using Thurston's method. 
Consider two  maps which are similar in some sense, for example, they have the same combinatorics (see as usual Section~\ref{Fiber Section}).
 In this case it may be possible to 
construct a  topological conjugacy on $\C$, and the goal is to upgrade this topological conjugacy to a quasiconformal conjugacy. One   can use pull backs to sacrifice the topological conjugacy in some areas (here, some control on the postsingular set is needed, and the 'similarity' of the two maps is used) to obtain a sequence of maps $\phi_n$ which are no longer conjugacies, but which are  uniformly quasiconformal and converge (with some extra care) to a map $\phi$. If one is careful enough with keeping the same homotopy class of $\phi$ for all the  $\phi_n$,   one can show that the limit map $\phi$ is a \emph{ Thurston Conjugacy}. A Thurston conjugacy  is not an actual conjugacy, but a quasiconformal map homotopic to a topological conjugacy  relative the postcritical set. However, a  theorem by Thurston and Sullivan \cite{Su88}, \cite[Lemma 4.3]{Che10} then ensure the existence of a true quasiconformal conjugacy.

A similar pullback argument starts with a quasiconformal homeomorphism that is a  conjugacy on some forward invariant set, for example a finite  graph formed by rays and their landing points.   One then can pull back this initial map (again, some control on the postsingular set is needed, and the 'similarity' of the two maps is used) to obtain a sequence of maps $\phi_n$ which are uniformly quasiconformal and converge (with some extra care) to a map $\phi$ which turns out to be a conjugacy  because it satisfies some appropriate functional equation.

% Consider two  maps which are similar in some sense, for example, they have the same combinatorics (see as usual Section~\ref{Fiber Section}). 
%One starts with a quasiconformal homeomorphism that is a  conjugacy on some forward invariant set,or with a topological conjugacy on $\C$. The forward invariant set can be for example a finite  graph formed by rays and their landing points, or the postsingular set. One then can pull back this initial map (here, some control on the postsingular set is needed, and the 'similarity' of the two maps is used) to obtain a sequence of maps $\phi_n$ which are uniformly quasiconformal and converge (with some extra care) to a map $\phi$ which turns out to be a conjugacy either because it satisfies some appropriate functional equation, or because it is a Thurston Conjugacy (which, in itslef,  is not an actual conjugacy, but a quasiconformal map homotopic to a topological conjugacy  relative the postcritical set. A theorem by Thurston and Sullivan \cite{Su88}, \cite[Lemma 4.3]{Che10} then ensure the existence of a true quasiconformal conjugacy).

%Using the Measurable Riemann Mapping Theorem one can see that quasiconformal conjugacy classes are either open or singletons. If one can say for example that the class of maps which are combinatorially equivalent to a given map is closed, one obtains that it is indeed a singleton. 
{See \cite{KSvS07} for another example in which the conjugacy on the boundaries of puzzle pieces given by the Boettcher map is exploited to construct a quasi-conformal conjugacy between two non-recurrent maps.}
 
Let us conclude this section with  a visual summary of the relations between the different classes   from \cite[Section 4]{Lyu99} (see Section~\ref{Fiber Section} for a definition of combinatorial classes).

Let $f$ be a unicritical polynomial or an exponential map. 
\[\text{Comb}(f)\supset  \text{Top}(f)\supset  \text{Qc}(f)\supset \text{Conf}(f)=\{f\}. \]
By $\text{Comb}(f)$ we mean the combinatorial class of $f$, that is %the connected component containing $f$ of 
all maps that are combinatorially equivalent to $f$; by $\text{Top}(f)$ (resp. $ \text{Qc}(f)$ , resp. $\text{Conf}(f)$), the connected component containing $f$ of the set of maps which are topologically (resp. quasiconformally, resp. conformally) conjugate to $f$. Observe that in the exponential family and more in general, two maps in the same family can be conformally conjugate (for example,  $e^z+c$ and $e^z+c+2\pi i $), but  conformal classes are still  singletons.

\section{The No Invariant Line Fields Conjecture (NILF)}\label{The No Invariant Line Fields Conjecture (NILF)}
% Indeed, before the work by Buff and Ch\'eritat one was tempted to believe that it could be possible to exclude the existence of invariant line fields on the Julia set by excluding the existence of Julia sets of positive measure. % proved the existence of Julia sets of positive measure [are they non-locally conneted?]. 
The relation  between MLC and the No Invariant Line Fields Conjecture is very well explained in \cite{McM}, \cite{McM94}. Since then, there has been a major breakthrough; namely,  the construction of examples of quadratic polynomials with Julia sets of positive measure 
by Buff and Ch\'eritat \cite{BC05}. 
One of the reasons why the result by Buff and Ch\'eritat received a lot of attention is because it was hoped that one could disprove the existence of  Julia sets with positive measure, and deduce that there could be no invariant line fields which in turn would imply density of hyperbolicity. Despite the existence of these Julia sets of positive measure, for now no invariant line fields have been constructed which are supported on the Julia sets of quadratic polynomials, so the problem is still open.  

\begin{defn}[\cite{McM}, Section 3.5] Let $f$ be a rational map or a transcendental map. We say that $f$ admits an \emph{invariant line fields} if there is a measurable Beltrami differential $\mu$ such that $f^*\mu=\mu$ almost everywhere, $|\mu|=1$ on a set $X$ of positive measure and $|\mu|=0$ elsewhere. 
\end{defn}
The name 'line field' is due to the fact that on the set where  $|\mu|=1$ the Beltrami differential defines a direction via  the field of ellipses induced by $\mu$.

%{ A \emph{line field} on a set $E\subset \C$ is a choice of a real  line through the origin in the tangent space for each $z\in E$ (see section 3.5 in \cite{McM}). This gives a real parameter $\theta(z)$, corresponding to the direction of the line, and induces a Beltrami differential
% $\mu(z):= e^{2\pi i \theta(z)} \frac{d\ov{z}}{dz}$. A line field is \emph{invariant} under $f$ if    $f^*\mu=\mu$, and $|\mu|=1$ on a set of positive measure. }

One is most interested in the case in which $X\subseteq J$.  
In particular, Julia sets which have measure zero cannot support line fields.

\begin{Conjecture}[No Invariant Line Fields Conjecture]\label{NILF conjecture} 
Suppose that  $f$ is a polynomial, an entire transcendental function with finitely many singular values or a  rational map which is not double covered by an integral torus endomorphism. Then   $f$   supports no invariant line field on its Julia set.  
\end{Conjecture}

 A   rational map which is not double covered by an integral torus endomorphism is called a flexible Latt\`es map. It admits the invariant line field generated by $\frac{d\ov{z}}{dz}$ on its Julia set (which is the entire Riemann sphere). 

An interesting feature of the No Invariant Line Fields Conjecture is that it shifts the focus from the structure of the parameter space to the ergodic properties of an individual  map, which are in principle easier to investigate.

Several results are known about absence of line fields for different classes of maps, one can find a summary in   the paper \cite{YZ10}, where the authors  use puzzle techniques to show that a rational map carries no invariant line field when its Julia set is a Cantor set. Another important result is a rigidity theorem by McMullen \cite[Theorem 10.2]{McM} stating that there are no invariant line fields for infinitely renormalizable quadratic polynomial-like maps which satisfy a priori bounds (the proof also works for degree $d$ unicritical polynomial maps \cite{Che10}). 
Additional results for absence of invariant line fields for rational maps, obtained using quadratic differentials, can be found in \cite{Ma05}. {If the modulus of the derivative on the  critical orbit grows fast enough,   the Julia set of a unicritical polynomial has no measurable invariant line field by \cite{Lev02}.}
% 
%It is not hard to see that when quasiconformal conjugacy classes are not points,  one can construct  invariant line fields  supported on the Julia set.
%{
%\begin{prop} Let $\FF$ be any of the families under consideration in this paper. If $f\in\FF$  is quasiconformally conjugate to $f'\in \FF$   then $f$ supports an  invariant line field. In particular, Conjecture~\ref{NILF conjecture} implies quasi-conformal rigidity. 
%\end{prop}
%{\red You need to make sure that that the ILF is not supported in the Fatou set.}
%\begin{proof}
%Let $\psi$ be a   quasiconformal conjugacy between $f$ and $f'$.  Let $\sigma_0$ be the standard conformal structure, that is  a field of circles, in the $f$-plane. Then $\psi^{*}\sigma_0$ is a field of ellipses in the $f'$-plane, which is not a field of circles  (otherwise  $\psi$ would be  conformal   by Weyl's Lemma).
%Since $\psi$ is a conjugacy, this induces an invariant fields of ellipses (or equivalently, an invariant lines field) for $f'$. 
%\end{proof} 
%}

The Density of hyperbolicity conjecture is equivalent to the NILF conjecture  (see \cite{MSS}, \cite{McM94}).
\begin{thm}[Non-existence of  Invariant Line Fields is equivalent to  Density of Hyperbolicity]\label{No ILF vs Density of Hyperbolicity} Let $\FF$ be a family of unicritical  polynomials or the exponential family.
A parameter $c$ belongs to a non-hyperbolic component if and only if $J_c$ has positive measure and supports an invariant line field. In particular, Conjecture \ref{Density of hyperbolicity conjecture} holds if and only if Conjecture \ref{NILF conjecture} holds.
\end{thm}
 
\begin{proof}[Sketch of the proof]
Let $f_c\in\FF$. Recall that the Julia set is the boundary of the set of escaping points $I(f_c)$(\cite{Mi}), and in the exponential case, $I(f_c)\subset J(f_c)$. If $c$ belongs to a non-hyperbolic component $U$, by Sullivan's classification of Fatou components, the   Fatou set is either empty  (if $f_c$ is an exponential) or coincides with the basin of infinity (if $f_c$ is a polynomial).
In both cases there is a natural holomorphic motion $\phi_\lambda$  of the set of escaping points over a neighborhood of $f_c$. For polynomials, one can simply use Boettcher coordinates; for exponential maps, one needs to recall that   components of the structurally stable set do not intersect the set of escaping parameters, so that dynamic rays move holomorphically over a non-hyperbolic component. By definition, this holomorphic motion is a conjugacy on $I(f_c)$.
By  the $\lambda$-Lemma (see Lemma~\ref{Lambda Lemma}), $\phi_\lambda$ extends as a quasiconformal  holomorphic motion (still denoted by $\phi_\lambda$) of  the closure $\ov{I(f)}=\C$. The map
$\phi_\lambda$ cannot be conformal on all of $\C$ (for $\lambda\neq c$ in a neighborhood of $c$), and in the polynomial case, it is conformal on $I(f)$ by construction. So the  dilatation $\mu_\lambda=\frac{\partial_{\ov{z}} \phi_\lambda}{\partial_{{z}} \phi_\lambda}$ is supported on $J(f_c)$ which needs to have positive measure. Since $\phi_\lambda$ is a conjugacy on $I(f))$ it is also a conjugacy on its closure, so $\mu:=\frac{\mu_\lambda}{|\mu_\lambda|}$ defines an invariant line field for $f_c$.

To prove the other implication let   $\mu$ be an invariant line field for $f_c$. By the Measurable Riemann Mapping Theorem there exists a holomorphic family of quasiconformal maps $\phi_t:\C\ra\C$ such that  
$\frac{\partial_{\ov{z}} \phi_t}{\partial_{{z}} \phi_t}=t\mu$ for $t\in\D$. Since $\mu$ is $f_c$-invariant, the maps $g_t:=\phi_t\circ f_c\circ\phi_t^{-1}$ are holomorphic (invariance implies that the standard conformal structure on $\C$  where all ellipses are in fact circles is mapped back to itself, hence all $g_t$ are holomorphic by Weyl's Lemma).
Since the $\phi_t$ are homeomorphisms, $g_t$ has the same number and degrees of critical points as $f_c$, hence is a unicritical polynomial of the same degree or an exponential map of the form $z^d+c(t), \ e^z+c(t)$ respectively. It can be checked that $c(t)$ is injective in $t$ so that this gives a neighborhood $U\ni c$ of maps which are quasiconformally (hence topologically) conjugate to $f_c$ hence $U\subset\SS(\FF)$. Since the Julia set of hyperbolic maps has measure zero, $U$ is contained in a non-hyperbolic component. 
\end{proof}

\begin{rem}For transcendental functions there are many examples for which the Julia set is the entire complex plane, and  the escaping set often has positive measure.  So it is more subtle to decide whether we want to disprove the existence of ILF on the Julia set or on the Julia set intersected the set of escaping points.   For entire functions with bounded set of singular values there are no invariant line fields supported on the escaping set by Theorem 1.2 in \cite{Re09}.  However, if one drops the assumption that the set of singular values is bounded  it is possible to  construct entire functions with an infinite-dimensional family of measurable invariant line fields  supported on the Julia set intersected the escaping set (see \cite[Example 5]{EL87}). 

A theorem about  absence of invariant line fields for a large class class of entire functions (under the dynamical assumptions that the postsingular set is bounded and a specific type of  non-recurrence)  can be found in  \cite{RvS11}. Additional results for transcendental functions with bounded postsingular set can be found in \cite{GKS}.\end{rem}

\section{ Renormalization and  Branner-Hubbard-Yoccoz Puzzle}\label{Renormalization}
Let $f\in\MM$ be a map in some appropriate space of functions.
In a rather general setting, to \emph{renormalize} means to consider the first return map $f^n$ to an appropriate   subset $X$ of the target space for $f$ (in our case,   $X\subset\C$) and then rescale or normalize $f^n|_X$ in an appropriate way so as to obtain another map $\RR f$ - called the \emph{renormalization} of $f$- which belongs in some sense to  the same class as $f$. The operator $\RR$ is often called the renormalization operator. The goal is to find a class $\MM$ of functions which is invariant under $\RR$ (although, $\RR$ may be partially defined on $\MM$) and so that $\RR$ is hyperbolic (i.e., strictly contracts distances on $\MM$). In this way the iteration of $\RR$ leads to convergence towards some map called the fixed point of the renormalization. 
There are many different types or renormalization which are relevant in the complex  polynomial setting; the reader can find a list in the introduction of  \cite{IS}. 

\subsection{Polynomial-like Renormalization}

The polynomial-like renormalization is a procedure which starts with a polynomial (or, a rational map, or, a transcendental map), associates to it a map in the much larger class of  polynomial-like maps, and then uses Douady-Hubbard Straightening Theorem to go back from the polynomial-like map to another true polynomial, which in general will be different from the initial one.
 One of the most beautiful applications of this polynomial-like renormalization is that it can be used to explain the presence of the small Mandelbrot copies inside the Mandelbrot set (see \cite{DH85}, \cite{EE85} for the original proofs, \cite{Lyu07} for an expository account, see also \cite{Lyu99}) and inside of the parameter spaces of other families of rational maps (\cite{McM00}).  There is a chapter on renormalization in \cite{McM} and in \cite{Lyu99}; more details can be found in the papers quoted in this section.

\begin{defn} 
Let $U, V$ be simply connected open sets with $U\Subset V$.
A \emph{polynomial like map} $g:U\ra V$ of degree $d$ is a holomorphic covering of degree $d$ from $U$ to $V$. 
\end{defn}

Polynomial-like maps have been introduced in \cite{DH85} to explain the self-similarity of the Mandelbrot set, that is  the presence of copies of $M$ inside itself which are quasiconformally equivalent to the entire Mandelbrot set. %[check DH, McM and Misha's book, 3 theorems by JC Yoccoz].  
Polynomial-like maps  share many features with true polynomial maps, for example it is possible to define the filled Julia set and  dynamic rays.  A fundamental theorem is the Douady-Hubbard Straightening Theorem. It states that every polynomial-like map of degree $d$ is quasi-conformally conjugate to a true polynomial (uniquely defined if the Julia set of the polynomial-like map is connected), and moreover that the conjugacy is conformal on the filled Julia set.  The straightening  can be seen as a natural projection from the space of all polynomial-like maps of degree $d$ to a specific slice consisting of  polynomials of degree $d$.

The following definition is from \cite{McM}, Chapter 7.  See also \cite{Lyu99}, 
\begin{defn}A polynomial  is called \emph{renormalizable} if there exist topological disks $U,U'$ such that $f^n:U\ra U'$ is polynomial-like  with connected Julia set and $U$ contains a critical point. 
\end{defn}

The polynomial-like map $f^n|_U$ can be straightened to a true polynomial, which  can possibly be renormalized again. If the renormalization procedure can be repeated infinitely many times $f$ is called \emph{infinitely renormalizable}. The Julia set of the renormalized map is  called the \emph{small Julia set}.

\begin{rem}There are several slightly different ways of defining a renormalization. Sometimes the renormalized map is just the restriction of $f^n$ to an appropriate neighborhood, other times the renormalized map is the polynomial arising from Doaudy and Hubbard's straightening theorem, other times it is rescaled or normalized in different ways. One can drop the   requirement that  the   Julia set for the polynomial-like map is connected, but then there is no uniqueness for the straightened polynomial. 
\end{rem}

As mentioned in the beginning, the process of renormalization is a dynamical system by itself, in that the \emph{renormalization operator} $\RR:f\ra Rf$ acts on the space of polynomial-like maps. Of course, it is only partially defined. 

Suppose that $f^m$ is polynomial-like in a neighborhood $U$ of a critical point, hence that $f$ is renormalizable with period $m$. Its renormalization can be of \emph{satellite} or \emph{primitive} type. In the satellite case, the small Julia set $J$ of the renormalization $f^m|_U$  and its images $f^i(J)$ for $i=0\ldots m$ all touch in a repelling fixed point $\alpha$ with some rotation number $p/q$, and are permuted according to this rotation number; in the primitive case, the small Julia set $J$ and its images are disjoint.

\begin{rem}Caution: In the renormalization context, the word \emph{combinatorics} usually refers to the following two sets of data: If the map is infinitely renormalizable of satellite type, its combinatorics is given by the sequence $\{p_i/q_i\}$ of the combinatorial rotation numbers under which the small Julia sets are permuted; if  
the map is  infinitely renormalizable of primitive type,  its combinatorics is given by the periods of renormalization together with additional homotopy information on the position of the small Julia sets. This is not the same  way in which the word combinatorics is used in Section~\ref{Fiber Section}. The red thread connecting these notions is the following: all these a priori different sets of data- the sets of renormalization data for satellite or for primitive renormalization, as well as the data from the orbit portrait in Section~\ref{Fiber Section},  identify a unique infinite sequence of  small Mandelbrot copies, and in all cases the goal of the  rigidity conjecture in any of its forms is to show that the diameters of these copies shrink to zero so that their limit is a unique point.  
\end{rem}

\begin{thm}[Yoccoz Theorem  \cite{Hu}, \cite{Mi00}]\label{Yoccoz Theorem}  Let  $z^2+c$  be at most finitely renormalizable with all periodic points repelling. Then $J_c$ is locally connected and the Mandelbrot set is locally connected at $c$.
\end{thm}
A very clear  outline of the proof of Yoccoz's Theorem can be found in \cite{McM94}. We remark that 
it took more than twenty additional years to prove the analog of Yoccoz's Theorem for unicritical polynomials of arbitrary degree $d$ (\cite{AKLS}; see also  \cite{KvS09}, \cite{PT15}).

\begin{thm}[AKLS] Let  $z^d+c$  be at most finitely renormalizable with all periodic points repelling.  Then $f_c$ is combinatorially rigid and the corresponding  Multibrot set is locally connected at $c$.
\end{thm} 
%
%The authors prove that any unicritical polynomial $f_c(z)=z^d+c$ = fc: z 7→ z
%d + c, d ≥ 2, which is at
%most finitely renormalizable and has only repelling periodic points, is combinatorially
%rigid. This implies that the connectedness locus (the “Multibrot set” Md) is locally
%connected at the corresponding parameter values and generalizes Yoccoz’s Theorem for
%quadratics to the higher degree case
 
\begin{rem}[A priori bounds]
%\subsection{A priori bounds}
Complex bounds have been introduced in \cite{Su88} and have been widely used thereafter (see \cite{Lev11} for a list of works using complex bounds).
Roughly speaking, a priori bounds imply that a sequence of renormalizations is precompact, allowing  the usage of a well established machinery based on Sullivan-Thurston pullback argument to obtain combinatorial  rigidity (see for example \cite{Che10}). So in many cases, the main difficulty when studying a new combinatorial class of (infinitely renormalizable) parameters relies consists in proving some kind of a priori bounds.  A priori bounds for several different combinatorial classes of parameters have been proven in \cite{Kah06}, \cite{KL08}, \cite{KL09}. 
It is worth noting that not all parameters satisfy some type of a priori bounds, for which alternative approaches have to be considered (\cite{Lev09}, \cite{Lev11}, \cite{Lev14}).
\end{rem}

For quadratic polynomials, Yoccoz's Theorem can be recovered from one of the main results of a recent work by Graczyk and \'Swi\c{a}tek (see Theorem 1 and p.56 in \cite{GS17}. Their work studies and uses properties of the harmonic measure on the boundary of the connectedness locus for the families of unicritical polynomials in order to deduce conformal similarity results between   parameter and dynamical space.  
 \subsection{Puzzles}
 Puzzles have been introduced for cubic polynomials by \cite{BH92}, and used by Yoccoz to prove theorem~\ref{Yoccoz Theorem}. In general, one starts with a forward invariant graph $\Gamma$ and then defines puzzle pieces of level $n$ as the connected components of $\C\setminus f^{-n}(\Gamma)$. By forward invariance of  $\Gamma$, puzzle pieces are either nested or disjoint, and puzzle pieces of a given level $n$ map to puzzle pieces of level $n-1$.  For polynomials $\Gamma $ usually consists of one or more cycles of periodic rays together with their landing points. Puzzles have been used to prove many rigidity and local connectivity results for polynomials. The idea behind this is that one considers a nested sequence of puzzle pieces of increasing levels (for polynomials, puzzle pieces  can be made into bounded sets by cutting them with equipotentials) containing a given point $z_0$ (usually, a critical point). One then estimates the modulus of the annulus between two consecutive puzzle pieces. If one can get that the sum over all the moduli is infinite, one gets that the intersection of all those pieces is the single point $z_0$; since the boundaries of puzzle pieces are escaping points (and preperiodic points, but this ends up being not relevant) this  gives a basis of connected neighborhoods of $z_0$ in the Julia set, proving local connectivity. Similar techniques are used to construct puzzles in the parameter space, which are called parapuzzles and are constructed using parameter rays. In fact,  puzzles are strictly related to fibers (see Section~\ref{Fiber Section}). 
This kind of 'controlling the modulus' techniques, based on the so-called \emph{a priori bounds}, are very involved, and are used for example in \cite{Lyu97} \cite{Kah06} \cite{KL08}, \cite{KL09}.

Puzzles for rational maps have been introduced in \cite{Ro}, see also \cite{YZ10}.
Puzzles can be used also to construct quasiconformal conjugacies between maps which have the same combinatorics (the same ray portrait), see \cite{Cui01},\cite{Be15}.

 %[what is the difference between fibers and puzzles? The nesting stuff is the same, but the way in which they are used is different.]

\subsection{Parabolic and Near-Parabolic Renormalization}

%{\color{Violet} Davoud: In 2006, Inou and Shishikura introduced a renormalization scheme that provides a powerful tool
%to study the dynamics of near parabolic maps, [IS06]. This involves an infinite-dimensional class of
%maps F, and a nonlinear operator R : F → F, called near-parabolic renormalization. Every map
%in F is  , has a neutral fixed point at 0, and a unique critical
%point of local degree two in its domain of definition. Given f ∈ F, R(f) is defined as a sophisticated
%notion of the return map of f about 0 to a region in the domain of f, viewed in a certain canonically
%defined coordinate on that region.}

Yet another quite different type of renormalization is the renormalization that can be done for parabolic or near-parabolic parameteters. It is  called \emph{parabolic} and \emph{near-parabolic renormalization} respectively, the latter being called also cylinder renormalization or Inou-Shishikura renormalization. They are both clearly described in \cite{IS}, for germs of the form $f(z)=z+a z^2+\ldots$ and $a\neq 0$. Very roughly speaking the renormalization of a parabolic   map (resp. a near parabolic map) is given by an iterate of $f$ which starts in a fundamental domain in the repelling petal and then goes back to the attracting petal (resp. starts in a fundamental domain in the repelling petal and comes back to the repelling petal). This is achieved using the horn map, or Lavaurs map. The appropriate iterate of $f$ is then associated to a map defined in a neighborhood of $0$ of the same form as $f$.

One of the main problems in this case is finding an invariant compact class where the renormalization procedure can be iterated infinitely many times (\cite{Sh98}, \cite{IS}). For parabolic renormalization, one such class is the class $\FF_0$ of appropriately normalized holomorphic functions defined in a Jordan  neighborhood of $0$, for which  $0 $ is an indifferent fixed point, and  which are branched covering with a unique singular value and all critical points of degree at most $2$. For this reason, this class is most suitable to study parabolic renormalization of quadratic polynomials. For near-parabolic renormalization the definition is more involved \cite{IS}.

Near-parabolic renormalization has been applied to the study of rigidity for quadratic polynomials which are  infinitely satellite  renormalizable. For this type of renormalization class the  nests of annuli near the critical point can become more and more degenerate (this is called lack of \emph{a priori bounds}; see  \cite{So00} for examples of satellite infinitely renormalizable polinomials with no a priori bounds). Indeed, under certain conditions on the sequence $\{p_i/q_i\}$ associated to the infinitely many polynomial-like renormalization of a map $f$ they are able to show that the map $f$ is also infinitely renormalizable in the near-parabolic sense, and that the sets of parameters in $\C$ which are $n$-times near-parabolic renormalizable have shrinking diameters.   The main result in this direction    is the following  \cite[Theorem D]{CS}:
\begin{theorem}\label{ChSh} Let $\ZZ=\{p_i/q_i\}$ be a sequence of rational numbers which satisfy appropriate conditions. Let $M_n(\ZZ)$ be the sets of parameters in the Mandelbrot set which are at least $n$ times satellite  renormalizable with combinatorics $p_1/q_1\ldots p_n/q_n$. Then  
there are constants $N, C$, and $\lambda \in(0, 1)$ such that 
$\diam M_n(\ZZ)\leq C λ^n$.  
In particular, if $f$ is a quadratic polynomial which  infinitely renormalizable of satellite type  with combinatorics $\{p_i/q_i\}$ it  is combinatorially rigid, and the Mandelbrot set is locally connected at c.
\end{theorem}
The condition required on $\{p_i/q_i\}$ in Theorem~\ref{ChSh} (as well as the methods used) is different from the condition used in \cite{Lev11}, \cite{Lev14} to obtain combinatorial rigidity for other classes of satellite infinitely renormalizable quadratic polynomials with non-locally connected Julia sets. 

In \cite{CS} the reader can also find some description of the relation between polynomial-like and near parabolic renormalization.

 \subsection{Renormalization and puzzles for transcendental maps}
  It can happen that a transcendental map $f$ is  polynomial-like in a neighborhood $U$ of one of its critical points. % For example one can find naturally defined polynomial-like maps enclosing a Siegel disk \cite{BF3}, which can be turned into true polynomial-like maps. 
   When a transcendental map is locally polynomial-like, it is possible that   some of the techniques from the polynomial case may be applied to study the structure of the dynamical plane; however  it is not evident right now  how to then apply those techniques to the study of the corresponding non-unidimensional  parameter spaces. It is even less clear how to define renormalization when $f:U\ra U'$ is an infinite degree unbranched covering, for example, how to define an 'exponential-like' map.  
 McMullen \cite{McM00} has shown  that small Mandelbrot sets are dense in the bifurcation locus for any holomorphic family of rational maps. Such a striking  phenomenon is referred to as 'universality' and is deeply connected with quadratic-like maps and renormalization.   The existence of regions where some transcendental maps are polynomial-like makes it natural to ask   in which  parameter spaces one can find appropriate slices which sport copies of the Mandelbrot set.  Copies of the Mandelbrot set can be found for example in the parameter space of the complex standard family \cite{Fa95}.

Puzzles for transcendental maps have   been used   to construct quasiconformal conjugacies for non-recurrent exponential maps \cite{Be15}. For the exponential family, the structure of ray portraits is very similar to the structure of ray portraits for polynomials, so cycles of  periodic  rays   give a natural forward invariant set to start the puzzle construction.  An important difference with respect to  the polynomial case  is that in many cases for functions with finitely many singular values, path connected components of the escaping set are just the rays, so that there is no natural way of cutting puzzle pieces to obtain bounded tiles and be able to use the geometric techniques to estimate the moduli of annuli. in conclusion it is possible to use the puzzle machinery to obtain combinatorial results and construct conjugacies but the geometry is unlikely to work out. 

\section{Combinatorial Rigidity and Triviality of Fibers}\label{Fiber Section}
\label{Combinatorial rigidity section}

In this section we will state two conjectures which are equivalent to MLC for the quadratic family, and which can be  stated in a natural way also for the families of unicritical polynomials and for the exponential family. 

Throughout this section, let $f\in\FF$ which is either the exponential family or a family of unicritical polynomials. Dynamic rays for $f$ are labeled by sequences in the  space $\SS$, where $\SS=\{-d/2+1,\ldots,0,\ldots d/2\}^\N$ if $d$ is even, $\SS=\{(-d-1)/2,\ldots,0,\ldots (d+1)/2\}^\N$ if $d$ is odd and $\SS=\Z^\N$ if $f$ is an exponential map.  For unicritical polynomials, the sets $\SS$ correspond to the p-adic expansion of the angle $\theta$, using symbols in a symmetric way with respect to $0$.

If the singular values is non-escaping, all periodic rays land at repelling or parabolic periodic points (\cite{DH85}, \cite{Re06}). Conversely, to any repelling or parabolic periodic point  $P$ we can  associate the $n$-tuple $O_P=\{G_{\s_1}\ldots G_{\s_n}\}$ of  rays which land at $P$. In the exponential case, it is not yet known whether $O_P$ could be empty for some $P$ in the case that $f$ has unbounded postsingular set.

The countable  collection $O_f=\{O_P, O_{P_2}\ldots\}$ of all $n$-tuples of rays landing at parabolic and repelling points is called the \emph{orbit portrait } for $f$. In this set we say that $g,f$ \emph{have the same combinatorics} if $O_f=O_g$.

\begin{rem}In \cite{Lyu97}, \cite{Che10}, the combinatorics of an infinitely renormalizable polynomial $P_c$ is defined as the sequence of maximal Mandelbrot copies which contain the parameters corresponding to the successive renormalization of $c$. This is indeed the same concept. One can translate the definition in \cite{Lyu97} to obtain a    sequence  of nested maximal Mandelbot copies which actually contain $c$, and each of which corresponds to a renormalization. This   defines a sequence of nested parabolic wakes containing $c$  which in turn determines which orbit portraits are realized by $P_c$: since orbit portraits are created and destroyed only when crossing wakes, for any wake  $W$  and any parameter $c\in W$ the map $P_c$ realizes all and only the portraits created while crossing the wakes containing $W$. There is a hierarchical structure of wakes given by which wakes contains which other wake.  Similarly, the combinatorics in terms of orbit portraits gives the sequence of all wakes in which the parameter $c$ is contained, a subsequence of which corresponds to the periodic rays identifying the (possibly finite) sequence of maximal Mandelbrot copies which contains $c$. The subsequence can be identified by the orbit portrait of $f$ which are contained in the sector containing $0$ for all other portraits. 
One of the advantages of defining combinatorics using orbit portraits is that it makes perfect sense also for transcendental maps (for which rays exist and desirably land). Although for example in the space of exponential maps one can visually and combinatorially identify subcopies of the bifurcation locus,  it is not at all clear how to prove that such copies are homeomorphic, and until now there has been no success in defining an  exponential analogue of polynomial-like maps.
\end{rem}

A map $f\in \FF$ is \emph{combinatorially rigid} if all its cycles are repelling and no other $g\in\FF$ has the same combinatorics.
 
\begin{Conjecture}[Combinatorial Rigidity]\label{Combinatorial Rigidity Conjecture} Let $\FF$ be the exponential family or a family of unicritical polynomials. 
Two functions $f_c,f_{c'}\in \FF$ with all periodic points repelling have the same orbit portrait if and only if $c=c'$.
\end{Conjecture}

Another version of Conjecture \ref{Combinatorial Rigidity Conjecture} as originally stated in \cite{Mc94}, \cite{Lyu97} is the following: 
\begin{Conjecture}[Combinatorial Rigidity version 2]\label{Combinatorial Rigidity Conjecture 2} 
Let $\FF$ be the exponential family or a family of unicritical polynomials. Any two combinatorially equivalent functions $f_c,f_{c'}\in \FF$ are quasiconformally conjugate.
\end{Conjecture}

 \begin{rem}Quasiconformal conjugacy does not imply conformal conjugacy in itself. 
 However, maps  which are  quasiconformally (in fact, even  topologically) conjugate must have the same combinatorics, %(By Boettcher's Theorem, or holomorphic motion of dynamic rays in the exponential case)
 so combinatorial classes always contain quasiconformal classes. At this point, if one can show that   combinatorially equivalent maps are always  quasiconformally conjugate, one gets that combinatorial classes and quasiconformal classes coincide. Since the latter are either open or singletons, one can use an open-closed argument  to show that they are singletons (see later in this section).
\end{rem}
 
The assumptions of all periodic points repelling is necessary because two points on the boundary of the same hyperbolic component have the same orbit portrait. 
 To have    Conjecture~\ref{Combinatorial Rigidity Conjecture}  stated in terms of maps with all cycles repelling is sufficient to obtain that it   implies density of hyperbolicity because parameters in non-hyperbolic components have all periodic cycles repelling.

For unicritical polynomials the Pommerenke-Levin-Yoccoz inequality implies that Multibrot sets are locally connected at points on the boundaries of hyperbolic components, and  it is known that points with attracting and indifferent cycles always belong to the closure of a hyperbolic component.  There is no known valid version of the Pommerenke-Levin-Yoccoz inequality for exponentials which would lead to a similar result.

 Conjecture~\ref{Combinatorial Rigidity Conjecture} is false in the space of (non-unicritical) cubic polynomials. Indeed, Henriksen \cite{Hen} has given an example of two cubic polynomials which have the same orbit portrait but are not quasiconformally conjugate on their Julia sets.

\subsection{Fibers for unicritical polynomials and exponentials}

Let $X$ be a Multibrot set or the Julia set of a unicritical polynomial.
The idea behind showing local connectivity of $X$  used in Section~\ref{Renormalization} for Multibrot sets is to find a nested  sequence of sets around a point $x$ (the puzzle pieces), whose  intersection with $X$ is connected,   and whose diameter shrinks to zero. The boundaries of such puzzle pieces consist of periodic and preperiodic ray pairs and pieces of equipotentials. 
 It is a natural idea to look at this from a different perspective, and define the (dynamical or parameter) fiber of $x$ as the intersection over all possible puzzle pieces containing $x$ whose boundary is made of preperiodic ray pairs. 
Fibers have been introduced in \cite{Sc04}, where one can find a good introduction to this topic. Fibers have the nice properties that, for a point $c\in M_d$, triviality of the fiber of  $c$ implies local connectivity of $M_d$ at $c$, since it implies the existence of a basis of connected neighborhoods around $c$. Here we follow the more recent exposition from \cite{RS08}, which applies better to the exponential family.
 
A \emph{separation line} is a Jordan arc in parameter space, tending to infinity in both directions,   which either contains only hyperbolic and finitely many parabolic parameter, or  consists of exactly two  periodic or preperiodic parameter rays together with their common landing point. 

Two points are separated by a separation line $\gamma$ if and only if they belong to two different connected components of $\C\setminus\gamma$.
 
 The \emph{extended fiber} of a parameter $c$ is  s the set of   parameters which cannot be separated from $c$ via a separation line. 
The \emph{(parameter) fiber} (also called \emph{reduced fiber}) of a non-escaping parameter $c$ is the set of non-escaping parameters which cannot be separated from $c$ via a separation line. The fiber of $c$ is trivial if and only if it only contains the parameter $c$. Extended fibers are closed and connected.
  
\begin{Conjecture}[Triviality of fibers]\label{Triviality of fibers conjecture} Consider the family of unicritical polynomials of degree $d\geq2$ or the exponential family. 
Then all  fibers  are trivial. 
\end{Conjecture}
 
There are many different definitions of separation lines which give rise to very similar theories of fibers. For example, one can consider only preperiodic parameter rays  (although in this case, fibers of  parameters on the closure of hyperbolic components are not trivial, so the conjecture would need to be stated in a slightly different way) or just separating lines which contain only hyperbolic and parabolic parameters.  These  definitions eventually yield the same results provided that fibers of parameters on the boundaries of hyperbolic components are trivial.  For more on the definition and properties of fibers see \cite{RS08}.

Since separation lines cannot cross non-hyperbolic components, a non-hyperbolic component would be contained in a single fiber, hence Conjecture~\ref{Triviality of fibers conjecture} implies Conjecture~\ref{Density of hyperbolicity conjecture}. 

On the opposite side, to separate any two parameters $c,c'$ in the boundary of the same hyperbolic component $U$ is fairly easy. Observe that $\partial U$ is an ordered set (For quadratic polynomials, consider the parametrization of $U$ given by the multiplier of the attracting cycle, which extends to the boundary of $U$ inducing on $\partial U$ the order from the unit circle; for the exponential family and unicritical polynomials with $d>2$ the multiplier map is a covering, but still induces a cyclic or vertical order). Then take a parabolic point $p_1$ such that  $c<p_1<c'$ and one parabolic point such that $c'<p_1<c$. Any two periodic parameter rays landing at $p_1,p_2$ respectively, together with a curve in $U$, separate $c$ from $c'$.

 The following fact is the reason why triviality of fibers is also referred to as combinatorial rigidity. It holds for unicritical polynomials and exponentials.
 
 {
\begin{thm} Let $c'$ be a non-escaping parameter with  all periodic cycles repelling. Then $c'$ belongs to the fiber $F(c)$ of another parameter $c$ if and only if $c$ and $c'$ have the same combinatorics. In particular     Conjecture~\ref{Combinatorial Rigidity Conjecture} (Combinatorial Rigidity Conjecture) is equivalent to Conjecture~\ref{Triviality of fibers conjecture} (Triviality of fibers conjecture).
\end{thm}
\begin{proof}[Sketch of the proof]
Since new periodic ray  portraits are born and die when crossing parabolic  wakes in parameter space (see Theorem~\ref{Ray portraits are born in wakes}),  two {non-escaping} parameters $c,c'$ have the same orbit portrait if and only if every  wake containing $c$ contains $c'$ and viceversa, or if they belong to the closure of the same stable component (hyperbolic, or possibly non-hyperbolic). 
If $c'\in F(c)$, $c$ and $c'$ are not separated by any periodic  parameter ray pair, so they have the same combinatorics. On the other side if $c,c'$ have the same combinatorics they cannot be separated by any periodic or preperiodic ray pair, so if all periodic cycles are repelling, they belong to the same fiber. 

Triviality of fibers implies that the intersection of any sequence of nested wakes  contains at most one non-escaping parameters, so no  two parameters with all cycles repelling can have the same orbit portrait. On the other side, two parameters which have different portraits can be separated by two periodic rays: take an n-tuple $O_P$ which belongs to the portrait of, say, $c$ but not $c'$, consider the angles/addresses $\s_1,\s_2$ of the characteristic rays for $O_P$ (the rays which bound the sector containing the singular value), and the parabolic wake $W$ formed  by the parameter rays of angle $s_1,s_2$. Then $c\in W$ and $c'\notin W$.
\end{proof}

%\begin{thm} Conjecture~\ref{Combinatorial Rigidity Conjecture} is equivalent to Conjecture~\ref{Triviality of fibers conjecture}.
%\end{thm}
%\begin{proof}[Sketch of the proof]
%Since new periodic ray  portraits are born and die when crossing parabolic  wakes in parameter space,  two {non-escaping} parameters $c,c'$ have the same orbit portrait if and only if every  wake containing $c$ contains $c'$ and viceversa, or if they belong to the closure of the same stable component (hyperbolic, or possibly non-hyperbolic). Triviality of fibers implies that the intersection of any sequence of nested wakes  contains at most one non-escaping parameters, so no  two parameters with all cycles repelling can have the same orbit portrait. On the other side, two parameters which have different portraits can be separated by two periodic rays: take an n-tuple $O_P$ which belongs to the portrait of, say, $c$ but not $c'$, consider the angles/addresses $\s_1,\s_2$ of the characteristic rays for $O_P$ (the rays which bound the sector containing the singular value), and the parabolic wake $W$ formed  by the parameter rays of angle $s_1,s_2$. Then $c\in W$ and $c'\notin W$.
%\end{proof}
The previous theorem in fact shows that for maps with all periodic cycles repelling, fibers and combinatorial classes coincide. To see whether combinatorial classes and quasiconformal classes coincide one would need to show that combinatorial equivalence implies quasiconformal conjugacy. }

\begin{rem}Density of hyperbolicity is much weaker than triviality of fibers. Indeed, fibers may be non-trivial but contain for example only parameters in the bifurcation locus. If all fibers have empty interior, clearly there are no non-hyperbolic components.% and have empty interior.  
\end{rem}

 For quadratic polynomials,  MLC is equivalent to triviality of fibers (see \cite{DH84}, \cite{Lyu99}, Lecture 4 and  \cite{RS08}, Theorem 10).
\begin{thm}For quadratic polynomials, the MLC Conjecture is equivalent to Conjecture~\ref{Combinatorial Rigidity Conjecture} (Combinatorial Rigidity Conjecture) and hence  to Conjecture~\ref{Triviality of fibers conjecture} (Triviality of fibers Conjecture). 
\end{thm}
\begin{proof}[Sketch of the proof]
Assume triviality of fibers and consider $c\in M$ and a sequence of nested wakes $W_n\ni c$. Since the fiber of $c$ is trivial $\cap W_n\cap M=\{c\}$, hence the wakes $W_n$ cut by appropriate equipotentials form a basis of connected neighborhoods for $c$.
Now let $F(c)$ denote the fiber of a parameter $c$. If $M$ is locally connected by Carath\'eodory-Torhost's Theorem every parameter ray accumulating on some  $c'\in F(c)$ actually lands at $c'$. Since fibers are either trivial or uncountable (as they are full sets) and can contain the accumulation set of at most finitely many parameter rays (see e.g. \cite{RS08} for details)  $F(c) $ is finite hence equals $\{c\}$.
\end{proof}

\begin{rem} The fiber $F(c) $ of a parameter $c$  is trivial  if and only if all parameters in $F(c)$ are landing points of a parameter ray (for both polynomials and exponentials). Non-trivial fibers are uncountable,  but contain the accumulation set of at most finitely many rays, so not all parameters in a  non-trivial fiber  can be landing points of rays. A trivial fiber $F(c)=\{c\}$ contains the accumulation set of at least one parameter ray, which is forced to land at $c$.
\end{rem}

\begin{rem}  In \cite{RS08}  the \emph{extended fiber} of a parameter $c$ is the set of parameters which cannot be separated from $c$ via a separating line, while the term \emph{(reduced) fiber} is used for the extended fiber minus the set of dynamic rays (without their endpoints, which may be escaping, but are considered to belong to the reduced fiber). We avoided this issue by only defining fibers for non-escaping parameters (as for the Mandelbrot set) and taking the intersection with the set of non-escaping points.  Unlike the polynomial case, the set of non-escaping parameters is neither open nor close, so the extended fibers are closed but the  reduced fibers are neither  open nor closed. 
\end{rem}

Thinking in terms of fibers   allows to treat classes of parameters which can be defined combinatorially, and to use  a wide range of combinatorial tools. The rigidity conjecture generalizes naturally to the other one-parameter families like unicritical polynomials and the exponential family.
Triviality of fibers has been recently proven for   exponential maps which are either postsingularly finite or  combinatorially non-recurrent (\cite{Be11},\cite{Be15}).
Triviality of fibers for parabolic parameters is work in progress of the author with L. Rempe-Gillen.

\subsubsection*{Open-closed argument} {Here is a  famous open-closed argument modified so that it works also in the exponential case (for polynomials see e.g. \cite{Lyu99}, Lemma 4.6). 

Using  the measurable Riemann Mapping Theorem, one can see that quasiconformal classes are either open or singletons.  Combinatorial classes coincide with fibers and   are  closed by definition for polynomials, so whenever one can   show that combinatorial classes and quasiconformal classes coincide, one  obtains that the latter are singletons and hence combinatorial rigidity. 

The set of non-escaping parameters for the exponential family is not compact, so an open-closed argument for the exponential familily is slightly more involved, in that one needs to show that reduced fibers, i.e. combinatorial classes, are never open. This may seem like an intuitively clear fact, but one should remember that reduced  fibers are the intersection of a closed set (the extended fiber) with  the set of non-escaping parameter, which is neither open nor closed. So a priori an extended fiber could be for example   the closure of a non-hyperbolic component whose boundary points are all escaping (so the reduced fiber would be just the non-hyperbolic component itself, which is open), or contain infinitely many (closures of) non-hyperbolic components whose union forms an open set.   To our understanding it does not seem to follow directly from the arguments in \cite{RS08}, so we include a proof here. %Observe that it is not necessary to show that they are closed. %This seems intuitively clear, however to our understanding it does not seem to follow directly from the arguments in \cite{RS08}, so we include a proof here.

\begin{prop}\label{Closed fibers} For the exponential family, (reduced) fibers cannot be open.
\end{prop}

\begin{proof}
Either a reduced fiber has empty interior (in which case it is clearly not open) or  it contains a  hyperbolic component (which is impossible since the fibers of hyperbolic parameters are singletons) or it contains a non-hyperbolic component. In the latter case, the fiber contains also all the points on the  boundary of the non-hyperbolic component, since separation lines cannot cross the latter. One can show that non-hyperbolic components always have (a lot of) non-escaping parameters in the boundary \cite[Lemma 6.2]{Be15}. Such parameters are in the reduced fiber by definition, so it is  enough to show that any such a parameter $c$ is in the boundary of
 the reduced fiber and not in its interior. Indeed, $c$ is approximated by  Misiurewicz parameters because the latter are dense in the bifurcation locus, and Misiurewicz parameters are always the landing points of parameter rays with preperiodic address, hence cannot belong to the interior of any fiber; hence $c$ itself must be on the boundary of the reduced fiber and the reduced  fiber is not open. 
\end{proof}  

Like in the polynomial case, if one can   show that combinatorial classes and quasiconformal classes coincide, one  obtains that the latter are singletons and hence combinatorial rigidity. We observe that this argument does not exclude the possibility that a reduced fiber may contain for example  two non-hyperbolic components touching at a common point, or infinitely many non-hyperbolic components attached like a string of beads or clustered in some kind of bifurcation tree.

\subsection{Some thoughts about fibers for families in class $S$}

There are several issues in stating a rigidity conjecture %or a triviality of fiber conjecture 
for more general classes of transcendental functions. A basic one is that  it has not been proven yet that periodic and preperiodic rays land (except when they intersect a singular orbit) so that this issue may cause some indeterminacy in the definition of combinatorial classes (for example, it may happen that two functions have the same ray portrait because some of their rays do not land at all, but that the 'natural way' in which those rays would have been supposed to land is very different).  {This is related to the fact that there is no obvious generalization of wakes. Indeed, given a cycle of periodic rays, it moves holomorphically outside the set of point in $M_\FF$ for which a singular value hits a ray in the cycle, an image thereof, or one of the landing points. However, it is not at all clear that the set where the cycle moves holomorphically is open: for exponentials and polynomials, wakes exist because parameter rays land.}

Another issue is that it is not clear whether parameters in non-hyperbolic component only have repelling cycles, so one would need to be very careful about stating the conjecture in terms of all cycles being repelling.  

In higher dimensional parameter spaces the notion of separating lines does not make sense. One could still define the fiber of a map $f$ as the connected component containing $f$ of the set of all maps which have the same combinatorics as $f$, but it would be much harder to show properties of fibers without using the fact that they are obtained as nested intersections  of closed (not necessarily compact) sets.

Finally,   the result by Henriksen about the failure of combinatorial rigidity for cubic polynomials make it doubtful whether it is worth pursuing this strategy for families whose parameter spaces are not unidimensional. Henriken  construction of two non-conjugate cubic polynomials  which are combinatorially equivalent does not go through for transcendental maps for several reasons, however, it seems to rely deeply  on the fact that the parameter space has complex dimension strictly bigger than one, hence seems plausible that similar phenomena could occur for transcendental entire maps.

\subsection*{A table of the different conjectures orbiting around the Density of Hyperbolicity Conjecture}
We conclude with  a table of the different conjectures  for quadratic polynomials and the relations among them. 
\vspace{20pt}
\begin{center}

\begin{tabular}{ccccc}
Triviality of Fibers & $\Longleftrightarrow$ & MLC & $\Longleftrightarrow$ & Combinatorial Rigidity\\
& & &   & \\ 
 & &$\Big\Downarrow$ &   & \\
 & & &   & \\
 No Invariant Line Fields & $\Longleftrightarrow$ & Density of Hyperbolicity & $\Longleftarrow$ & Topological Rigidity\\
 & & &   & \\
& & $\Big\Downarrow$ &   & \\
& & &   & \\
& & QC Rigidity &   & 
\end{tabular}

\end{center}

\begin{footnotesize}

\end{footnotesize}

\end{document}